\documentclass{amsart}
\usepackage[T1]{fontenc}
\usepackage{a4wide}
\usepackage{amsthm}
\usepackage{amssymb}
\usepackage{graphicx}
\makeatletter
\numberwithin{equation}{section} 
\numberwithin{figure}{section} 
\theoremstyle{plain}
\theoremstyle{plain}
\newtheorem{thm}{Theorem}[section]
\theoremstyle{plain}
\newtheorem{prop}[thm]{Proposition}
\theoremstyle{remark}
\newtheorem{rem}[thm]{Remark}
\theoremstyle{plain}
\newtheorem{cor}[thm]{Corollary}
\theoremstyle{plain}
\newtheorem{lem}[thm]{Lemma}
\theoremstyle{plain}
\newtheorem{defn}[thm]{Definition}
\theoremstyle{plain}
\newtheorem*{propn}{Proposition}

\DeclareFontFamily{OT1}{rsfs}{}
\DeclareFontShape{OT1}{rsfs}{m}{n}{ <-7> rsfs5 <7-10> rsfs7 <10-> rsfs10}{}
\DeclareMathAlphabet{\mycal}{OT1}{rsfs}{m}{n}

\newcommand{\eps}{\varepsilon}
\newcommand{\BJ}[2]{J_{#1} \left(#2 \right)}
\newcommand{\BJp}[2]{J^\prime_{#1}\left(#2\right)}
\newcommand{\BY}[2]{Y_{#1} \left(#2 \right)}
\newcommand{\BYp}[2]{Y^\prime_{#1}\left(#2\right)}
\newcommand{\BH}[2]{H^{(1)}_{#1}\left(#2 \right)}
\newcommand{\BHp}[2]{H^{(1)\prime}_{#1}\left(#2 \right)}
\newcommand{\Rn}{R_n\left(\oeps,\lambda\right)}
\newcommand{\Ro}{R_0\left(\oeps,\lambda\right)}
\newcommand{\Tn}[2]{T_n\left(#1,#2\right)}
\newcommand{\oeps}{\omega_{\eps}}
\newcommand{\forget}[1]{ }
\newcommand{\lneps}{\left|\ln \eps\right|}
\def\dsp{\displaystyle}
\begin{document}

\title[Size Estimates]{On the scattered field generated by a ball inhomogeneity of constant index}

\author{Yves Capdeboscq}
\address{Mathematical Institute, 24-29 St Giles, OXFORD OX1 3LB, UK}

\begin{abstract}
We consider the solution of a scalar Helmholtz equation where the potential (or index) takes two positive values,
one inside a disk of radius $\eps$ and another one outside.  We derive sharp estimates of the size of the scattered field
caused by this disk inhomogeneity, for any frequencies and any contrast. We also provide a broadband estimate,
that is, a uniform bound for the scattered field for any contrast, and any frequencies outside of a set which 
tends to zero with $\eps$.
\end{abstract}

\maketitle

\section{Introduction}

We consider a scalar field satisfying the Helmholtz equation with frequency $\omega>0$ in $\mathbb{R}^2$.
Given a prescribed incident field $u^i$, a non-singular solution of 
\begin{equation}\label{eq:eq-intro-1}
\Delta u^{i}+\omega^{2}q_{0}u^{i}=0 \mbox{ in } \mathbb{R}^{2},
\end{equation}
we are interested in the solution $u_{\eps}\in H^1_{\mbox{loc}}\left(\mathbb{R}^{2}\right)$ of 
\begin{equation}\label{eq:eq-intro-2}
\Delta u_{\eps}+\omega^{2}q_{\eps}u_{\eps}=0\mbox{ in }\mathbb{R}^{2},
\end{equation}
where, for $|x|>\eps$, $u_{\eps}=u^{i}+u_{\eps}^{s}$, and  $q_{\eps}$ equals $q$ 
inside the inhomogeneity and $q_{0}$ outside. 
We take the inhomogeneity to be a disk of radius $\eps$. The coordinate system is chosen so that the 
inhomogeneity is centered at the origin. In other words
\[
q_{\eps}(r):=\left\{ \begin{array}{ll}
q & \mbox{ if }r<\eps\\
q_{0} & \mbox{ if }r>\eps\end{array}\right.
\]
We assume that both $q_{0}$ and $q$ are real and positive. We assume that the scattered field satisfies the classical Silver-M\"uller \cite{MULLER-69, NEDELEC-01} outgoing radiation condition, 
given by
\begin{equation}\label{eq:eq-intro-3}
\frac{\partial}{\partial r}u_{\eps}^{s}-i\omega\sqrt{q_{0}}u_{\eps}^{s}=o\left(\frac{1}{\sqrt{r}}\right),
\end{equation}
where, as usual $r:=|x|$. 
Altogether, the conditions (\ref{eq:eq-intro-1},\ref{eq:eq-intro-2},\ref{eq:eq-intro-3}) 
imply that the incident field $u^i$, the scattered field $u_\eps^s$ and the transmitted field $u_\eps^{t}=u_\eps$ for $r<\eps$, 
admit series expansions in terms of special functions, namely
\begin{eqnarray}
u^{i}(x)&\sim&
\sum\limits_{n=-\infty}^{\infty}a_{n} \BJ{n}{\sqrt{q_{0}}\omega r}\exp\left(i\,n\arctan\left(\frac{x}{r}\right)\right), \label{eq:uinc}\\
u_\eps^{s}(x)&\sim&
\sum\limits_{n=-\infty}^{\infty}a_{n}  \Rn\BH{n}{\sqrt{q_{0}}\omega r}\exp\left(i\,n\arctan\left(\frac{x}{r}\right)\right), \label{eq:usca}\\
u_\eps^{t}(x)&\sim&
\sum\limits_{n=-\infty}^{\infty}a_{n}  \Tn{\oeps}{\lambda} \BJ{n}{\sqrt{q}\omega r}\exp\left(i\,n\arctan\left(\frac{x}{r}\right)\right). \label{eq:utra}
\end{eqnarray}
In the above formulae, $\BJ{n}{x}=\Re(\BH{n}{x})$, and  $x\to\BH{n}{x}$  is the Hankel function of the first kind of order $n$. The rescaled non-dimensional frequency 
$\oeps$, and the contrast factor $\lambda$ are given by
\begin{equation}\label{eq:def-oeps-lambda}
 \oeps:= \sqrt{q_0}\omega\eps \, \mbox{ and } \lambda := \sqrt{\frac{q}{q_0}}.
\end{equation}
The reflection and transmission coefficients $R_n$ and $T_n$ are given by the transmission problem on the boundary of the inhomogeneity, that is, at $r=\eps$.
They are the unique solutions of
\begin{eqnarray*}
\Tn{\oeps}{\lambda}  \BJ{n}{\lambda \oeps} 
&=& 
\BJ{n}{\oeps} + \Rn \BH{n}{\oeps}, \\
\lambda  \Tn{\oeps}{\lambda}  \BJp{n}{\lambda\oeps} &=&  
 \BJp{n}{\oeps} +  \Rn  \BHp{n}{\oeps},
\end{eqnarray*}
which are
\begin{equation}\label{eq:def-RN}
\Rn=- \frac{\displaystyle \Re\left(\BHp{n}{\oeps}\BJ{n}{\lambda\oeps}
-\lambda\BJp{n}{\lambda\oeps}\BH{n}{\oeps}
\right)
 }{\displaystyle  
 \BHp{n}{\oeps}\BJ{n}{\lambda\oeps}
-\lambda\BJp{n}{\lambda\oeps}\BH{n}{\oeps}
},
\end{equation}
and, after a simplification using the Wronskian identity satisfied by $\BJ{n}{\cdot}$ and $\BH{n}{\cdot}$,
\begin{equation}\label{eq:def-TN}
\Tn{\oeps}{\lambda}=\frac{2 i}{\pi}\frac{1}{\displaystyle  
\BHp{n}{\oeps}\BJ{n}{\lambda\oeps}
-\lambda\BJp{n}{\lambda\oeps}\BH{n}{\oeps}
}.
\end{equation}
It is well known that both $R_n$ and $T_n$ are well defined for all $\lambda>0$ and $\oeps>0$, see e.g. \cite{HANSEN-POIGNARD-VOGELIUS-07} for a proof. Note that $R_n=R_{-n}$,
and $T_n=T_{-n}$ for all $n$.

In \eqref{eq:uinc}, \eqref{eq:usca} and \eqref{eq:utra}, the $\sim$ symbol is an equality if the right-hand-side is replaced by its real part, the fields being real. 
By a common abuse of notations, in what follows we will identify $u^i$ and $u_\eps^s$ with the full complex right-hand-side.

Such expansions have been known for almost two centuries. They allow in principle, with the help of modern computers and recent numerical methods,  to compute 
the scattered field accurately, given the incident field $\omega$, $\eps$ and $q/q_0$. Yet, they do not give any insight on the behavior of the scattered field when 
the frequency, the contrast, or the radius $\eps$ vary. When $\eps$ tends to zero, the behavior of the scattered field for this problem has been  
studied recently in \cite{HANSEN-POIGNARD-VOGELIUS-07}.  The cases considered are either $a_{n}=0$ for $n>N_{0}$, 
or $\Im(q)>0$, or full reflection on the boundary of the inclusion, that  is, $u_{\eps}=0$ at $r=\eps$.  In this work, we focus on non lossy inclusions, 
that is, when $\Im(q)=0$, and we provide sharp estimates of the scattered field. These estimates are derived for any sequence $(a_n)$, thus 
for any incident field. They are completely explicit, up to the numerical values of the constants involved.  Such detailed results are 
possible because of the extensive studies of Hankel functions conducted by others. We will quote frequently the classical treatise of Watson \cite{WATSON-25}, 
and we will indirectly refer to the book of Olver \cite{OLVER-74} by frequently citing the NIST Handbook of Mathematical Functions \cite{NIST-10}. Other papers 
related to properties of Bessel functions \cite{BOYD-DUNSTER-86,LANDAU-99,LANDAU-00,MULDOON-SPIGLER-84,PARIS-84,SZASZ-50,THIRUNAN-51} are also cited in the proofs. 
Some additional estimates that we could not find in the literature are provided in Appendix~\ref{ap:A}. 
Some of them could be new, but we have not performed a comprehensive search of the vast literature on that topic. 
However with the exception of  Section~\ref{sec:quasi-res}, our main results are stated in a form that does not require any knowledge of the literature 
related to Bessel functions, except possibly for some universal constants (approximate numerical values are provided).

Let us now discuss the norms we shall use. Given any $f \in C^0(\mathbb R^2)$, its restriction to the circle $|x|=R$ is a periodic function. 
We can therefore define its complex Fourier coefficients 
$$
c_n\left(f(|x|=R)\right) = \int_0^{2\pi} f(R,\theta) e^{- i n \theta } d\theta,
$$
and $f(|x|=R)$ can be measured in terms of the following Sobolev norm
\begin{equation}\label{eq:norm-hs-0}
\left\Vert  f \left(|x|=R\right) \right\Vert _{H^{\sigma}} := 
\sqrt{2\pi} \sqrt{\sum_{n=-\infty}^{\infty} |c_n\left(f(|x|=R)\right)|^2 (1+|n|)^{2\sigma} },
\end{equation}
for any real parameter $\sigma$. By density, this norm can be defined for less regular functions. If $f(|x|=R)$ is $L^2(0,2\pi)$ for example 
it is bounded, for any $\sigma\leq0$. To measure the oscillations of $f$ only, we will use
\begin{equation}\label{eq:norm-hs-0-star}
\left\Vert  f\left(|x|=R\right) \right\Vert _{H_{*}^{\sigma}} 
:=  
\left\Vert  f\left(|x|=R\right) - \frac{1}{2\pi}\int\limits_0^{2\pi} f\left(|x|=R \right) d\theta \right\Vert _{H^{\sigma}}.
\end{equation}
For radius independent estimates, we shall use the semi-norm
\begin{equation}\label{eq:norm-ns-0}
\mathcal{N}^{\sigma} (f) := \sqrt{2\pi} \sqrt{\sum_{n\neq0} \sup_{R>0} |c_n\left(f(|x|=R)\right)|^2 (1+|n|)^{2\sigma} }.
\end{equation}
It is easy to see that this norm is finite for a smooth $f$ with bounded radial variations. Finally, to document the sharpness of our estimates, 
we will provide lower bounds in terms of a semi-norm,
\begin{equation}\label{eq:norm-np-0}
\mathbf{N}_{p}^{\sigma} \left( f \right) :=  \sqrt{2\pi} \sup\limits_{|n| \geq p} \sup_{R>0} |c_n\left(f(|x|=R)\right)| (1+|n|)^{\sigma},
 \end{equation} 
where $p$ is a positive parameter. These norms are satisfy the following inequality
$$
\left\Vert  f \left(|x|=R\right) \right\Vert _{H^{\sigma}} \leq \mathcal{N}^{\sigma} (f), \mbox{ and } \mathbf{N}_{p}^{\sigma} \left( f \right) \leq \mathcal{N}^{\sigma} (f),
$$
and if for all $R$, $f\left(|x|=R\right)$ only has one non-zero Fourier coefficient,
$$
\mathbf{N}_{1}^{\sigma} \left( f \right) =  \mathcal{N}^{\sigma} (f) = \sup_{R>0} \left\Vert  f \left(|x|=R\right) \right\Vert _{H_{*}^{\sigma}}.
$$
We choose these three (semi-)norms $\left\Vert  \cdot \right\Vert _{H^{\sigma}}$, $\mathcal{N}^{\sigma}$ and $\mathbf{N}_p^{\sigma}$ because they are
compatible with expansions \eqref{eq:uinc}, \eqref{eq:usca} and \eqref{eq:utra}. For example,
\begin{equation}\label{eq:norm-hs}
\left\Vert  u_{\eps}^{s}\left(|x|=R\right) \right\Vert _{H^{\sigma}} :=   
\sqrt{2\pi} \left(\sum_{n=-\infty}^{\infty}\left| \Rn a_{n}\right|^2
\left(1+\left|n\right|\right)^{2\sigma}  \left|\BH{n}{\sqrt{q_{0}}\omega R}\right|^{2}\right)^{\frac{1}{2}},
\end{equation} 
and 
\begin{equation}\label{eq:norm-hs-n0}
\mathcal{N}^{\sigma} \left( u ^{i} \right) :=  \sqrt{2\pi} \left(\sum_{n \neq 0 }\left|a_{n}\right|^{2} \sup_{x>0}\left|\BJ{n}{x}\right|^2\left(1+\left|n\right|\right)^{2\sigma}\right)^{1/2}.
 \end{equation} 
Furthermore, it is known \cite{LANDAU-00} that for all $n\neq0$
\begin{equation}\label{eq:ineq-landau}
\frac{4}{7} \frac{1}{(|n|+1)^{1/3}} \leq \sup_{x>0} \left|\BJ{n}{x}\right| \leq \frac{6}{7} \frac{1}{(|n|+1)^{1/3}},
\end{equation}
therefore $\mathcal{N}^{\sigma} \left( u ^{i} \right)$ has upper and lower bounds depending on $a_n$ only, namely,
\begin{equation}\label{eq:ineq-landau-2}
\frac{8\pi}{7} \sum_{n \neq 0 }\left|a_{n}\right|^{2} \left(1+\left|n\right|\right)^{2\sigma -2/3} \leq \left(\mathcal{N}^{\sigma} \left( u ^{i} \right)\right)^2
\leq \frac{16\pi}{7} \sum_{n \neq 0 }\left|a_{n}\right|^{2} \left(1+\left|n\right|\right)^{2\sigma -2/3}.
\end{equation}

The motivation for this work comes from imaging.  In electrostatics, the small volume asymptotic expansion for a diametrically bounded conductivity 
inclusion is now well established, and the first order expansion has been shown to be valid for any contrast \cite{NGUYEN-VOGELIUS-09}. It is natural to 
wonder whether such expansion could also hold for non-zero frequencies, even in a simple case.  

Section~\ref{sec:lleq1} addresses the case $\lambda\leq 1$. 
In Theorem~\ref{thm:os-ls} we derive  perturbation-type estimates when $\lambda\oeps<1$, that is, proportional to $(\lambda-1)\oeps^2$ at first order, 
for all $x$ such that $|x|\geq \eps$. This can be seen as a generalization of the electrostatic case. We show that the range of frequencies for which 
this result applies is sharp. 
In Theorem~\ref{thm:ob-ls}, we provide an upper estimate for the scattered field valid for all frequencies and all $|x|\geq \eps$, 
and we document its sharpness by providing a lower bound  for the supremum of the scattered field for all frequencies.
Section~\ref{sec:lgeq1}, \ref{sec:quasi-res}, \ref{sec:ffield} address the case $\lambda\geq1$.
Theorem~\ref{thm:os-lb} is similar to Theorem~\ref{thm:os-ls} and applies when $\lambda\sqrt{\ln \lambda +1}\,\oeps<1$, and when $\lambda\,\oeps<1$ if 
there is no zero-order term.
In Section~\ref{sec:quasi-res}, we provide a detailed study of quasi-resonances. These are frequencies located just after the perturbative range, at which 
the near-field becomes arbitrarily large. Theorem~\ref{thm:ob-lr} provides lower bounds for the near field in this regime. We also provide numerical examples 
of quasi-resonant modes.
Section~\ref{sec:ffield} provides far field estimates, that is, for $x$ such that $|x| \lambda >\eps$, valid for all frequencies. As in Theorem~\ref{thm:os-lb}, 
we show that the bounds provided are sharp.

Another inspiration for this work is recent results concerning the so-called cloaking-by-mapping method for the Helmholtz equation. 
In \cite{KOHN-ONOFREI-VOGELIUS-WEINSTEIN-10}, the authors show that cloaks can be constructed using lossy layers, and that non-lossy 
media could not be made invisible to some particular frequencies (the quasi-resonant frequencies). In Section~\ref{sec:bband}, we show in 
Lemma~\ref{lem:highcontrast-Lemma} that if an interval around these frequencies is removed, contrast independent estimates for the near-field
can be obtained. 
When $\lambda=\eps^{-1}$, the following proposition is proved as a corollary of Lemma~\ref{lem:highcontrast-Lemma}.
\begin{propn}
Assume $\eps<1/7$, and $\lambda=\eps^{-1}$. Then, for any $\alpha,\beta>0$, there exists a set $I_1$ depending on $\eps,\alpha$ and $\beta$ and a set $I_0$ depending
on $\eps$ and $\beta$ which satisfies
$$
|I_1| \leq \eps^{\beta} \lneps,\quad |I_0|\leq  \frac{\ln \lneps}{\left(\lneps+1\right)^{\beta}},
$$
such that for all $R\geq\eps$,
$$
\sup\limits_{\sqrt{q_0}\omega \in(0,\infty)\,\setminus \, I_1} \left\Vert  u_{\eps}^{s}\left(|x|=R\right) \right\Vert^2 _{H_{*}^{\sigma}}  
\leq \frac{18}{\alpha} \sqrt{\frac{\eps^{1-2\beta}}{R}}\mathcal{N}^{\sigma + 2 + \alpha}\left(u^i\right),
$$
and
$$
\sup\limits_{\sqrt{q_0}\omega \in(0,\infty)\,\setminus \, I_0} \left|\frac{1}{2\pi}\int\limits_0^{2\pi} u_{\eps}^s\left(|x|=R \right)\right|  \leq   
12 \frac{1}{\sqrt{\left(\lneps+1\right)^{3/2 -2\beta} R}}. 
$$
\end{propn}
We do not prove that this result is sharp. Combining Lemma~\ref{lem:highcontrast-Lemma} with results of the previous sections, we show in 
Theorem~\ref{thm:broadband} that broadband estimates uniform with respect to the contrast are possible. In particular we show that 
when observed at any fixed distance $|x|=R>0$, the scattered field vanishes in the limit  $\eps=0$ except in a set of frequencies  of 
zero measure.

Section~\ref{sec:caseone} is devoted to the proof of intermediate estimates stated without proofs in Section~\ref{sec:lleq1} and \ref{sec:lgeq1}. 
Section~\ref{sec:case2} is devoted the proof of an intermediate estimate used in Lemma~\ref{lem:highcontrast-Lemma}.

\section{\label{sec:lleq1}Inclusions with relative index smaller than one}

This section is devoted to the case when $q<q_0$.  We estimate the scattered field at a distance $R\geq\eps$ from the center of the inclusion. 
Our first result addresses the case of moderate frequencies.
\begin{thm}\label{thm:os-ls} 
Let $y_{0,1}$ be the first positive solution of $\BY{0}{x}=0$. When $\eps\leq R$, $\lambda\leq1$, and 
 $\oeps < y_{0,1}$, there holds
$$ 
\left\Vert u_{\eps}^{s} \left(|x|=R\right) \right\Vert _{H^{\sigma}}\leq  
\left(1- \lambda \right)\oeps \left( 3 \sqrt{\frac{\eps}{R}} \left\Vert  u^{i}\left(|x|=\eps\right) \right\Vert _{H_{*}^{\sigma-1/3}}
+ 9 \, \oeps  \left|u^{i}(0)\right|  \left|\BH{0}{\sqrt{q_0}\omega R} \right|\right).
$$
Furthermore, if for some $p>0$ the first $p$ Fourier coefficients of $u^{i} \left(|x|=(\eps\omega)^{-1}\right)$ are zero, then for all $\oeps<p$ there holds
$$ 
\left\Vert u_{\eps}^{s} \left(|x|=R\right) \right\Vert _{H^{\sigma}}\leq  
3 \left(1- \lambda \right)\oeps \sqrt{\frac{\eps}{R}} \left\Vert  u^{i}\left(|x|=\eps\right) \right\Vert _{H_{*}^{\sigma- 1/3}}.
$$
\end{thm}
To compare Theorem~\ref{thm:os-ls} with known results, we derive the following variant.
\begin{cor}\label{cor:thm-sosl}
When $\eps\leq R$, $\lambda\leq1$ and $\oeps < y_{0,1}$ we have 
$$ 
\left\Vert u_{\eps}^{s} \left(|x|=R\right) \right\Vert _{H^{\sigma}}\leq  9
\left(1- \lambda \right)\oeps^2 \left( \left|u^{i}(0)\right|  \left|\BH{0}{\sqrt{q_0}\omega R} \right| + \sqrt{\frac{\eps}{R}} \mathcal{N}^{\sigma-\frac{1}{3}}
\left( u^{i}\right)\right).
$$
\end{cor}
\begin{rem} 
Under this form, one can read for example that the first order term in $\eps$ is correct both with respect to $\eps$ and with respect to the contrast.
First order asymptotic expansions for small volume or small contrast perturbations \cite{AMMARI-VOGELIUS-VOLKOV-01,AMMARI-KANG-04} derived for a fixed frequency are
of the order of $\eps^2$. Note that this estimate holds up to frequencies of the order $\eps^{-1}$: this shows that the inhomogeneity can be viewed as a 
perturbation up to frequencies of that order.
\end{rem}

Our second result is an estimate valid for all frequencies.

\begin{thm}\label{thm:ob-ls} 
When $\eps\leq R$ and $\lambda\leq1$ there holds
\begin{equation}\label{eq:bd-up-ob-ls-2}
\sup\limits_{\omega>0} \left\Vert u_{\eps}^{s} \left(|x|=R\right) \right\Vert _{H^{\sigma}} 
\leq 
\frac{5}{2} 
\sqrt{\frac{\eps}{R}} \mathcal{N}^{\sigma} \left( u ^{i} \right)
 + \sqrt{2\pi}\left|u ^{i}(0)\right| \left|\BH{0}{y_{0,1} \frac{R}{\eps}}\right|.
\end{equation}
Furthermore,
\begin{equation}\label{eq:bd-down-ob-ls}
\sup\limits_{\omega>0} \left\Vert  u_{\eps}^{s} \left(|x|=\eps \right) \right\Vert _{H_{*}^{\sigma}} 
\geq 
\frac{1}{\sqrt{10}} \mathbf{N}_{n_0}^{\sigma} \left( u^{i} \right)  ,
\end{equation}
where $n_0$ is the smallest positive number such that 
$$
\lambda^2\leq1-\frac{49}{9n^{2/3}}. 
$$
for all $n\geq n_0$.
\end{thm}
\begin{rem}
The lower bound \eqref{eq:bd-down-ob-ls} shows that the upper bound \eqref{eq:bd-up-ob-ls-2} 
is sharp in the case when $\mathbf{N}_{n_0}^{\sigma}$ and $\mathcal{N}^\sigma$ are equivalent norms. 
We give a more precise upper bound in remark~\ref{rem:opti-small-2}.
\end{rem}

The proof of these results is based on a careful study of $\Rn$ conducted in Section~\ref{sec:caseone}. 
We prove that the following proposition holds
\begin{prop}\label{pro:lleq1}
Let $y_{n,1}$ be the first positive solution of $\BY{n}{x}=0$. When $\lambda\leq 1$, there holds, for all $n\geq 1$
\begin{itemize}
 \item For all $\oeps < y_{n,1}$,
$$
\left|\Rn \BH{n}{\oeps}\right| \leq \frac{5}{2} \left| \BJ{n}{\oeps}\right|.
$$
\item For all $\oeps < n$,
$$
\left|\Rn \BH{n}{\oeps}\right| \leq 2 (1 -\lambda) \frac{\oeps}{n^{1/3}}  \left| \BJ{n}{\oeps}\right|.
$$
\item when $\lambda^{2}<1-\left(\frac{7}{3\, n^{1/3}}\right)^{2}$, then 
$$
\left|R_n\left(n,\lambda\right) \BH{n}{n}\right|>\frac{1}{2}  \left| \BJ{n}{n}\right|.
$$
\end{itemize}
For $n=0$, there holds
\begin{itemize}
\item For all $\oeps>0$,
$$\left| \Ro \BH{0}{\oeps \frac{R}{\eps}} \right|  \leq  \left|\BH{0}{y_{0,1} \frac{R}{\eps}}\right|.$$
\item When $x<y_{0,1}$,
$$
\left| \Ro \BH{0}{x \frac{R}{\eps}} \right| \leq   \frac{\pi^2}{2\sqrt{2}} (1-\lambda)  \oeps^2 \left|\BH{0}{\oeps \frac{R}{\eps}}\right|. 
$$
\end{itemize}
\end{prop}
\begin{proof}
 This follows from Lemma~\ref{lem:FirstCase} and Proposition~\ref{pro:R0}.
\end{proof}

\begin{proof}[Proof of Theorem~\ref{thm:os-ls}]
From formula~\eqref{eq:norm-hs}, we have
\begin{align*}
\left\Vert  u_{\eps}^{s}\left(|x|=R\right) \right\Vert^2 _{H^{\sigma}} =&   
2\pi \sum_{|n|\neq 0} \left|a_{n}\right|^2 \left(1+\left|n\right|\right)^{2\sigma} 
 \left|\Rn |\BH{n}{\oeps}\right|^2  \left|\frac{\BH{n}{\sqrt{q_{0}}\omega R}}{\BH{n}{\oeps}}\right|^{2} \\
&+ 2\pi \left|a_{0}\right|^2 \left|\Ro\BH{0}{\sqrt{q_{0}}\omega R}\right|^{2}.
\end{align*}
Note that $y_{0,1}<1$. Proposition~\ref{pro:lleq1} shows that when $\oeps \leq n$ and $n\geq1$,
$$
\left|\Rn \BH{n}{\oeps}\right| 
\leq 2             \left(1- \lambda\right) \frac{\oeps}{n^{1/3}}      \left|\BJ{n}{\oeps}\right|
\leq 3             \left(1- \lambda\right) \frac{\oeps}{(1+|n|)^{1/3}}\left|\BJ{n}{\oeps}\right|,
$$
When $n\neq0$, $\sqrt{x}\left|H_{n}^{1}(x)\right|$ is decreasing \cite[13.74]{WATSON-25}, therefore
$$
R \left|\BH{n}{\sqrt{q_{0}}\omega R}\right|^{2} \leq \eps \left|\BH{n}{\oeps}\right|^{2},
$$
On the other hand, $a_0 = u^{i}(0)$, and Proposition~\ref{pro:lleq1} shows that
$$
\left|\Ro\BH{0}{\sqrt{q_{0}}\omega R}\right| 
\leq  
\left(1- \lambda \right)\frac{\pi^2}{2\sqrt{2}}\oeps^2 \left|\BH{0}{\sqrt{q_0}\omega R}\right|
$$
Combining these estimates we have
\begin{align*}
\left\Vert  u_{\eps}^{s}\left(|x|=R\right) \right\Vert^2 _{H^{\sigma}} &\leq \left( 3 \left(1- \lambda \right)\oeps \sqrt{\frac{\eps}{R}} \right)^2   
2\pi \sum_{|n|\neq 0} \left|a_{n}\right|^2 \left(1+\left|n\right|\right)^{2\sigma-2/3} 
 \left|\BJ{n}{\oeps}\right|^2  \\
&+\left(  u^{i}(0)  \left(1- \lambda \right)\oeps^2  \right)^2   \frac{\pi^5}{4}  \left|\BH{0}{\sqrt{q_0}\omega R}\right|^2 \\
&= \left( 3 \left(1- \lambda \right)\oeps\right)^2 \left(
 \frac{\eps}{R} \left\Vert  u^{i}\left(|x|=\eps\right) \right\Vert _{H_{*}^{\sigma-\frac{1}{3}}}^2
+  9  \, \oeps^2 \left|u^{i}(0)\right|^2  \left|\BH{0}{\sqrt{q_0}\omega R} \right|^2\right),
\end{align*}
and the conclusion follows. 
If for some $p>0$ the first $p$ Fourier coefficients of $u_{\eps}^{s} \left(|x|=(\eps\omega)^{-1}\right)$ are zero, $a_n=0$ for $n=0,\ldots, p-1$, and the argument above 
proves our claim.
\end{proof}

\begin{proof}
[Proof of Corollary~\ref{cor:thm-sosl}]
For all $n\geq 1$, it is known that \cite{PARIS-84} for $0<x<y<n$,
$$
\BJ{n}{x} \leq \frac{x^n}{y^n} \BJ{n}{y} \exp \left( \frac{y^2-x^2}{2n+2}\right).
$$
In particular,
$$
\BJ{n}{\oeps} \leq \frac{\oeps}{y_{0,1}} \BJ{n}{y_{0,1}} \exp \left( \frac{y_{0,1}^2}{4}\right) \leq 3 \, \oeps \, \BJ{n}{y_{0,1}}.
$$
This implies that
$$
\left\Vert  u^{i}\left(|x|=\oeps\right) \right\Vert _{H_{*}^{\sigma-\frac{1}{3}}} 
\leq  3 \, \oeps \, \left\Vert  u^{i}\left(|x|=y_{0,1}\right) \right\Vert _{H_{*}^{\sigma-\frac{1}{3}}}.
$$
inserting this upper bound in the estimate provided by Theorem~\ref{thm:os-ls} proves our claim.
\end{proof}

\begin{proof}[Proof of Theorem~\ref{thm:ob-ls}]
Starting from the formula \eqref{eq:norm-hs}, 
using the monotonicity 
of $x \left|\BH{n}{x}\right|^{2}$ 
for $n\geq1$ as in the proof of Theorem~\ref{thm:os-ls}, we obtain
\begin{align*}
\left\Vert  u_{\eps}^{s}\left(|x|=R\right) \right\Vert^2 _{H^{\sigma}} 
&\leq&   
2\pi \sum_{|n|\neq 0} \left|a_{n}\right|^2 \left(1+\left|n\right|\right)^{2\sigma} 
 \left|\Rn |\BH{n}{\oeps}\right|^2  \frac{\eps}{R} \\
&&+ 2\pi \left|a_{0}\right|^2 \left|\Ro\BH{0}{\sqrt{q_{0}}\omega R}\right|^{2}.
\end{align*}
Thanks to Proposition~\ref{pro:lleq1},
\begin{equation}\label{eq:bd-pf-obls-0}
 \left|\Ro\BH{0}{\sqrt{q_{0}}\omega R}\right| \leq \left|\BH{0}{y_{0,1} \frac{R}{\eps}}\right|.
\end{equation}
and when $n\geq1$, and  $\oeps < y_{n,1}$,
\begin{equation}\label{eq:bd-pf-obls-1}
 \left|\Rn |\BH{n}{\oeps}\right| \leq \frac{5}{2} \left|\BJ{n}{\oeps}\right| \leq  \frac{5}{2} \sup_{x>0} \left|\BJ{n}{x}\right|.
\end{equation}
On the other hand, the definition of $R_n$ \eqref{eq:def-RN} shows that $\left|\Rn \right|\leq1$. 
It is known \cite[13.74]{WATSON-25} that for $x>n\geq0$, 
$$
x\to\sqrt{x^{2}-n^{2}}\left|H_{n}(x)\right|^2
$$
is an increasing function of $x$, with limit $2/\pi$. 
Consequently, since $y_{n,1}>n+ \frac{13}{14}n^{1/3}$, \cite[15.3]{WATSON-25}, \cite{NIST-10}, 
for all $\oeps \geq   y_{n,1}$ we have
\begin{equation}\label{eq:bd-pf-obls-3}
\left|\Rn \BH{n}{\oeps}\right| \leq\sqrt{\frac{2}{\pi}\frac{1}{\sqrt{\oeps^{2}-n^{2}}}}\leq \frac{4}{5} \frac{1}{(1+n)^{1/3}} 
\leq \frac{7}{5}  \sup_{x>0} \left|\BJ{n}{x}\right|,
\end{equation}
thanks to \eqref{eq:ineq-landau}.
Combining  \eqref{eq:bd-pf-obls-1}, and \eqref{eq:bd-pf-obls-3}  we have obtained that for all $\oeps>0$, 
$$
2\pi \sum_{|n|\neq 0} \left|a_{n}\right|^2 \left(1+\left|n\right|\right)^{2\sigma} 
 \left|\Rn |\BH{n}{\oeps}\right|^2  \leq \frac{5}{2} \mathcal{N}^{\sigma} \left( u ^{i} \right)^2,
$$
which concludes our proof of the upper bound \eqref{eq:bd-up-ob-ls-2}.

Turning to the lower bound, we have
$$
\sup\limits_{\omega>0} \left\Vert u_{\eps}^{s} \left(|x|=\eps\right) \right\Vert _{H_{*}^{\sigma}} 
 \geq   
\sup_{|n|\geq n_0} \sup\limits_{\omega>0}\sqrt{2\pi} \left|a_n \left(1+\left|n\right|\right)^\sigma \Rn \BH{n}{\oeps }\right|.
$$
We know from Proposition~\ref{pro:lleq1} that provided $\lambda^2 < 1 - \frac{49}{9n^{2/3}}$,
$$
 \left|R_n\left(n,\lambda\right) |\BH{n}{n}\right| \geq \frac{1}{2} \BJ{n}{n} 
$$
Since it is known \cite[8.54]{WATSON-25}  that $n\to n^{\frac{1}{3}}\BJ{n}{n}$ is increasing,
$$
\left|R_n\left(n,\lambda\right) \BH{n}{n}\right| > \frac{\BJ{1}{1}}{2n^{\frac{1}{3}}} \geq \frac{1}{\sqrt{10}} \sup_{x>0} \left| \BJ{n}{x} \right|,
$$
where in the last inequality we used a variant of \eqref{eq:ineq-landau}, \cite{LANDAU-00}.  Choosing $\omega$ such that $\oeps=n$, we obtain
$$
\sup\limits_{\omega>0} \left\Vert u_{\eps}^{s} \left(|x|=\eps\right) \right\Vert _{H_{*}^{\sigma}} 
\geq    \frac{1}{\sqrt{10}}   \mathbf{N}_{n_0}^{\sigma} \left( u^{i} \right),
$$
as announced.
\end{proof}

\begin{rem}\label{rem:opti-small-2}
The lower bound was obtained for $\oeps=n$, that is, in the special case when the order and argument are equal. 
This is precisely the upper limit for $\oeps$ in Theorem~\ref{thm:os-ls}.
When the argument is much larger than the order, one should expect a decay gain of $1/2$ and not $1/3$. The bound given 
by Theorem~\ref{thm:ob-lb} is of this form, and applies here also (when $\lambda$ is replaced by $1$).
\end{rem}

\section{\label{sec:lgeq1}Inclusions with relative index larger than one: the perturbative regime}

In this section, we consider the case when $q>q_0$ in the case of moderate frequencies and moderate contrast.
Our result is expressed in terms of a threshold $m_\lambda$ which depends on the contrast, given by
\begin{equation}\label{eq:mlambda}
m_{\lambda} := \frac{1}{\lambda \sqrt{\ln \lambda +1}}
\end{equation}

\begin{thm}\label{thm:os-lb}
Suppose $R \geq \eps$, and $\lambda\geq1$. When 
$$
\oeps < \min\left(\frac{1}{2}, m_{\lambda}\right)
$$
we have
$$ 
\left\Vert u_{\eps}^{s} \left(|x|=R\right) \right\Vert _{H^{\sigma}}\leq  
 \left(1- \lambda \right)\oeps \left( 3 \sqrt{\frac{\eps}{R}} \left\Vert  u^{i}\left(|x|=\eps\right) \right\Vert _{H_{*}^{\sigma-\frac{1}{3}}}
+ 23 \, \oeps \lambda \left|u^{i}(0)\right|  \left|\BH{0}{\sqrt{q_0}\omega R} \right|\right).
$$
Furthermore, if for some $p>0$ the first $p$ Fourier coefficients of $u_{\eps}^{s} \left(|x|=(\eps\omega)^{-1}\right)$ are zero, then for all $\oeps<\lambda^{-1} p$ 
there holds,
$$ 
\left\Vert u_{\eps}^{s} \left(|x|=R\right) \right\Vert _{H^{\sigma}}\leq  
3 \left(1- \lambda \right)\oeps \sqrt{\frac{\eps}{R}} \left\Vert  u^{i}\left(|x|=\eps\right) \right\Vert _{H_{*}^{\sigma-\frac{1}{3}}}
$$
\end{thm}
The proof of Theorem~\ref{thm:os-lb} is, {\it mutatis mutandis}, the same as that of Theorem~\ref{thm:os-ls}, using Proposition~\ref{pro:lgeq1}, 
proved in Section~\ref{sec:caseone} in lieu of Proposition~\ref{pro:lleq1}, and we omit it.
\begin{prop}\label{pro:lgeq1}
When $\lambda\geq 1$, there holds, for all $n\geq 1$
\begin{itemize}
\item For all $\oeps \leq \lambda^{-1} y_{n,1}^{(1)}$,
$$
\left|\Rn \BH{n}{\oeps}\right| \leq \frac{5}{2} \left| \BJ{n}{\oeps}\right|.
$$
\item For all $\oeps < \lambda^{-1}n$,
$$
\left|\Rn \BH{n}{\oeps}\right| \leq 2 (1 -\lambda) \frac{\oeps}{n^{1/3}}  \left| \BJ{n}{\oeps}\right|.
$$
\end{itemize}
For $n=0$, there holds
\begin{itemize}
\item For all $\oeps$ such that $\min\left(\frac{1}{2},m_{\lambda}\right)\oeps\leq 1$,
$$
\left| \Ro \BH{0}{x \frac{R}{\eps}} \right| \leq   \frac{5\pi^2}{4} (\lambda-1) \lambda  \oeps^2 \left|\BH{0}{\oeps \frac{R}{\eps}}\right|. 
$$
\item For all $\oeps$, we have
$$
\left| R_{0}^{\eps}\BH{0}{x \frac{R}{\eps}} \right| \leq \sqrt{5} \left|\BH{0}{\min\left(\frac{1}{2},m_\lambda\right) \frac{R}{\eps}} \right|.
$$
\end{itemize}
\end{prop}
\begin{proof}
 This is follows from Lemma~\ref{lem:FirstCase} and Proposition~\ref{pro:R0}. 
\end{proof}

In contrast with the case $\lambda\leq1$, the range of frequencies for which Theorem~\ref{thm:os-lb} is 
valid becomes increasingly small as the contrast increases. The two extreme contrast cases, $\lambda=0$ and $\lambda=\infty$ are therefore of a very different nature. As
we will see in Section~\ref{sec:quasi-res}, the range of frequencies for which Theorem~\ref{thm:os-ls} applies is sharp: the behavior of the near field is drastically 
different when $\oeps$ is larger.

\section{\label{sec:quasi-res}Inclusions with relative index larger than one: quasi-resonances}

In this section, we investigate the behavior of the scattered field when $q>q_0$, and the product of the effective frequency and the contrast $\lambda\oeps$ is bounded.
In such a case, a quasi-resonance phenomenon occurs: near the inclusion, the scattered field becomes extremely large, for some frequencies.
 We refer to such frequencies as quasi-resonant frequencies. They are defined in Definition~\ref{def:qr}. 

Before we proceed, we remind the readers of usual notations and known properties of zeros of Bessel functions.

When $n\geq 0$, the $k$-th positive root of $\BJ{n}{x}=0$ is written $j_{n,k}$. The first positive root of $\BY{n}{x}=0$ is written $y_{n,1}$. 
When $n\geq1$, the $k$-th positive solution of $\BJp{n}{x}=0$ is written $j_{n,k}^{(1)}$. When $n=0$, we count non-negative solutions, that is,  $j_{0,1}^{(1)}=0$.
The first positive root of $\BYp{n}{x}=0$ is noted $y_{n,1}^{(1)}$
It is known \cite[15.3]{WATSON-25} that 
\begin{align} 
\mbox{when } n\geq 1,\quad & j_{n,1}        = n+ a_{n,1}      n^{1/3}, \mbox{ with } a_{n,1}      >\lim_{n\to \infty}a_{n,1}       \approx 1.86, \nonumber\\
&j_{n,1}^{(1)}  = n+ a_{n,1}^{(1)}n^{1/3}, \mbox{ with } a_{n,1}^{(1)}                >\lim_{n\to \infty}a_{n,1}^{(1)} \approx0.81,\nonumber \\
&y_{n,1}        = n + b_{n,1}     n^{1/3}, \mbox{ with } b_{n,1}                      >\lim_{n\to \infty}b_{n,1}       \approx0.93, \label{eq:abn-wats}\\
&y_{n,1}^{(1)}  = n + b_{n,1}^{(1)} n^{1/3}, \mbox{ with } b_{n,1}^{(1)}              >\lim_{n\to \infty}b_{n,1}^{(1)} \approx 1.82, \nonumber\\
\mbox{when }n=0,\quad & j_{0,1}    \approx2.40, \quad j_{0,1}^{(1)} =0, \quad  y_{0,1} \approx 0.894 \quad y_{0,1}^{(1)} \approx 2.20.\nonumber
\end{align}
The zeros of $\BJ{n}{\cdot}$ and $\BY{n}{\cdot}$ are interlacing  \cite[10.21]{NIST-10} and we have 
\begin{equation}\label{eq:Dixon-JnYn}
n\leq j_{n,1}^{(1)} <y_{n,1}<y_{n,1}^{(1)}<j_{n,1}<\ldots<j_{n,k}^{(1)}<j_{n,k}<\ldots
\end{equation}
and the first inequality is strict when $n>0$. Using the notations 
\begin{equation}\label{eq:def-MnThetan}
M_{n}(x)=\sqrt{\BJ{n}{x}^{2}+\BY{n}{x}^{2}},\theta_{n}(x)=\arg(J_{n}(x)+iY_{n}(x)),
\end{equation}
we have \cite[10.18]{NIST-10} 
\begin{equation}\label{eq:prothetan}
\lim_{x\to 0+} \theta_{n}(x) = -\frac{\pi}{2},\quad \frac{d}{dx}  \theta_{n}(x)>0, 
\quad  \theta_{n}(x)\approx x -\frac{2n+1}{4} \pi \mbox{ for } x \mbox{ large}.
\end{equation}
This shows in particular that for $n$ fixed, the size of the intervals $(j_{n,1}^{(1)},j_{n,1})$ is strictly decreasing, and tends to $\pi/2$. 

\begin{defn}\label{def:qr}
For any $n\geq0$, the triplet $(n,x,\lambda)$ is called quasi-resonant if
$$
0<x< y_{n,1},
$$
and if the reflection coefficient given by \eqref{eq:def-RN} is of maximal amplitude, that is, 
$$
R_n\left(x,\lambda\right)=-1.
$$
\end{defn}
When $(n,\oeps,\lambda)$ is quasi-resonant, problem~\eqref{eq:eq-intro-2} has the following particular solution
\begin{equation}\label{eq:res-sol}
u_\eps =  
        \begin{cases}
         \frac{\BY{n}{\oeps}}{\BJ{n}{\lambda\oeps}} 
                \BJ{n}{\lambda\oeps \frac{|x|}{\eps} }\exp\left(i\,n \arctan\left(\frac{x}{|x|}\right)\right)& \mbox{ when } |x|\leq\eps \\
          \BY{n}{\oeps \frac{|x|}{\eps}}\exp\left(i\,n \arctan\left(\frac{x}{|x|}\right)\right) &\mbox{ when } |x|\geq\eps.
         \end{cases}
\end{equation}
Note that $u_\eps$ is not truly a resonance, since $\BY{n}{\cdot}$ does not satisfy the outgoing radiation condition. The solution $u_\eps$ 
contains an incident field given by
$$
u^{i} =  \BJ{n}{\oeps \frac{|x|}{\eps} }\exp\left(i\,n \arctan\left(\frac{x}{|x|}\right) + \frac{\pi}{2} \right).
$$
The almost resonant behavior of this solution is apparent in the near field. The amplitude of the incident field at $|x|=\eps$ is $\BJ{n}{\oeps}$, 
whereas the amplitude of the scattered field is given by $|\BH{n}{\oeps}|$. 
Suppose for example that  $\oeps\approx{n}{K}$, with $K>1$ fixed - as Proposition~\ref{pro:Dixon} below shows, this is the generic case.
Then at $|x|=\eps$ the amplitude of $u_\eps$ grows geometrically with $n$,  \cite[10.19]{NIST-10} 
$$
\lim_{n\to\infty}  |\BY{n}{\oeps}|^\frac{1}{n} = \left( K +\sqrt{K-1}\right) e^{\sqrt{1- \frac{1}{K}}},
$$
whereas the amplitude of the incident field and its normal derivative at decays with the inverse rate,
$$
\lim_{n\to\infty}  |\BJ{n}{\oeps}|^\frac{1}{n} = \lim_{n\to\infty}  |\frac{d}{dx}\BJ{n}{\oeps}|^\frac{1}{n} = \frac{1}{K +\sqrt{K-1}} e^{-\sqrt{1- \frac{1}{K}}}.
$$
The size of the scattered field is therefore not controlled by the size of the incident field: this  behavior can be compared to that of a resonant mode. 
Note that the amplitude of the scattered field is also large compared to the maximal amplitude of the incident field anywhere, as the uniform bound
\eqref{eq:ineq-landau-2} indicates. The lower bound for the maximal value of the incident field is the motivation from the restriction $\oeps< y_{n,1}$ in the definition of quasi-resonances. 
Indeed, when  $\oeps>y_{n,1}$, using the bound \cite[13.74]{WATSON-25} 
$$
\left|\BH{n}{\oeps}\right|\leq\sqrt{\frac{2}{\pi\sqrt{\oeps^2 -n^2}}}, 
$$
we obtain, for all $n\geq0$
\begin{equation}\label{eq:bd-bignx}
\left|\BH{n}{\oeps}\right|\leq \frac{7}{5} \sup_{x>0} \left|\BJ{n}{x} \right|,
\end{equation}
therefore the scattered field, and in turn the full field, is comparable to the maximal amplitude of the incident field in this regime.

The following variant of Dixon's Theorem on interlacing zeros \cite[15.23]{WATSON-25} proves the existence of quasi-resonances.
\begin{prop}\label{pro:Dixon}
For any $n\geq0$ and $\lambda > j_{n,1}/y_{n,1}$, in every interval
$$
U_{n,k}=\left(\frac{j_{n,k}^{(1)}}{\lambda} \frac{j_{n,k}}{\lambda}\right)
\mbox{such that } U_{n,k} \subset \left(\frac{j_{n,1}^{(1)}} {\lambda},y_{n,1}\right)
$$
there exists a unique frequency $\omega_{n,k}$ such that the triplet $(n,\omega_{n,k},\lambda)$ is quasi-resonant. 
There are no quasi-resonances in the interval $(0, j_{n,1}^{(1)}/\lambda)$ when $n\geq1$, or when $\lambda<j_{n,1}/y_{n,1}$.
\end{prop}

The proof is given at the end of this section. To illustrate this result, we consider the case when $\lambda=2$ and $n=30$. 
The quasi-resonances are to be found in the interval  $(j_{30,1}^{(1)}/2,y_{30,1})\approx (16.28,32,98)$. 
There are $8$  such frequencies. 
The first one is $\omega_{30,1}\approx 17.4211682$, and the last one is $\omega_{30,8}\approx 31.4683226$.
 Figure~\ref{fig:qr-2} shows two plots on a logarithmic scale. The red line shows the radial component of full field $u_\eps$, 
corresponding to a relative index $\lambda=2$, an effective frequency $\omega_{30,1}$. The blue line shows the radial component of 
the  incident field $u^i$,  $\BJ{30}{\omega_{30,1} \cdot}$.  Note that the blow-up region is concentrated around $\eps$. 
At $|x|=\lambda\eps=2\eps$, the full field and the  incident field are of the same order of magnitude. This is the far field regime 
discussed in Section~\ref{sec:ffield}.

\begin{figure}
 \begin{center}
  \includegraphics[width=10cm]{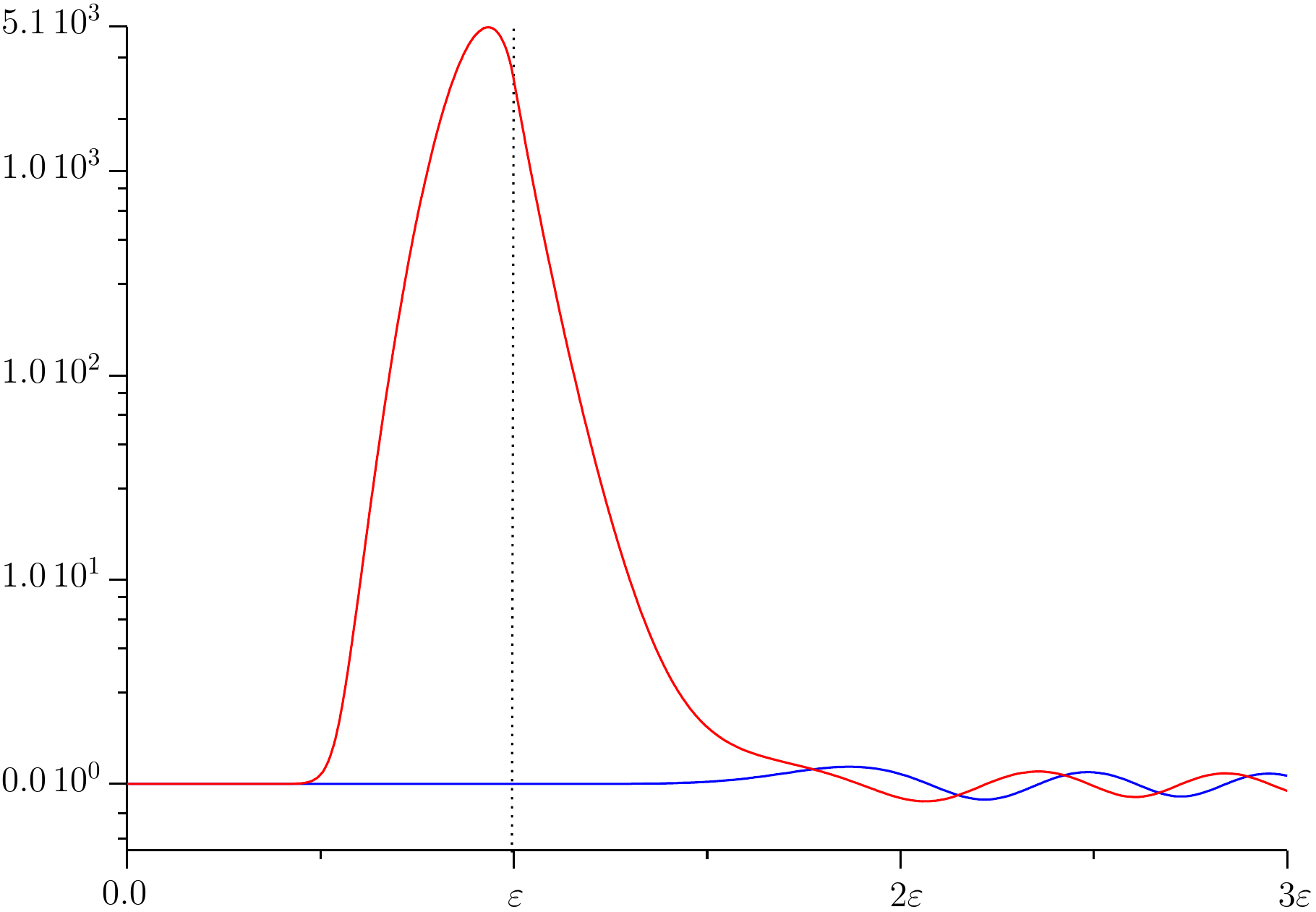}
 \end{center}
\caption{\label{fig:qr-2}First quasi-resonant solution for $\lambda=2$ and $n=30$.}
\end{figure}

Figure~\ref{fig:qr-3} shows a plot the radial component of full field $u_\eps$ in red, 
corresponding to a relative index $\lambda=2$, an effective frequency $\omega_{30,8}$, 
and the radial component of the incident field $u^i$, $\BJ{30}{\omega_{30,8} \cdot}$, in blue. 

This last quasi-resonance, situated close to the upper bound $y_{30,1}$, does not show a blow-up around $|x|=\eps$. This vindicates the choice to limit 
the definition of quasi-resonances to the interval $(0,y_{n,1})$.

\begin{figure}
 \begin{center}
  \includegraphics[width=10cm]{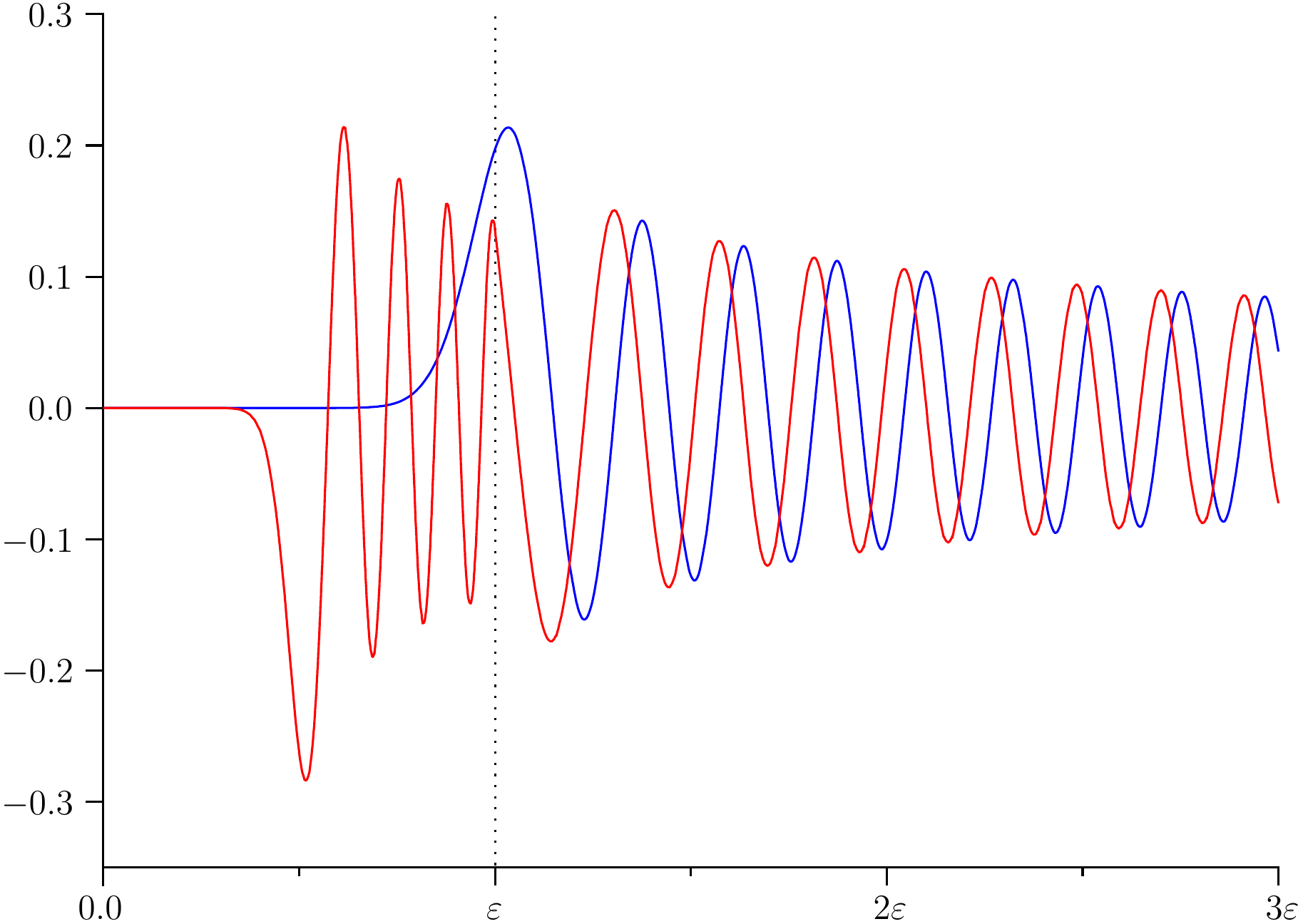}
 \end{center}
\caption{\label{fig:qr-3}Last quasi-resonant solution for $\lambda=2$ and $n=30$.}
\end{figure}

Quasi-resonances provide lower-bounds for frequency independent  scattering estimates, as the following Theorem shows.
\begin{thm}\label{thm:ob-lr}
Given $\lambda>1$, let $n_0$ be the smallest integer such that
$$
\lambda > \frac{j_{n_0,1}}{y_{n_0,1}}.
$$
Then, for any $p\geq n_0$,
\begin{equation}\label{eq:lowerbound-blowup-1}
\sup \left\{ \left\Vert u_{\eps}^{s} \left(|x|=R\right) \right\Vert _{H^{\sigma}} ,0<\lambda \oeps \leq j_{p,1}\right\}
\geq \sup\limits_{ n_0 \leq n \leq  p} \left( |a_{|n|}| (1+|n|)^\sigma \BH{n}{\frac{j_{n,1}}{\lambda} \frac{R}{\eps}} \right).
\end{equation}
Furthermore, if $\lambda >\exp(2)$,
\begin{eqnarray}\label{eq:lowerbound-blowup-0}
&& \sup \left\{ \left\Vert u_{\eps}^{s} \left(|x|=R\right) \right\Vert _{H^{\sigma}},
0< \lambda \oeps < \frac{\sqrt{2}}{\sqrt{\ln \lambda}}\left(1+\frac{1}{2\sqrt{\ln\lambda}}\right)\right\} \nonumber \\
&& \geq   |a_0|  \BH{0}{\frac{\sqrt{2}}{\lambda\sqrt{\ln \lambda}}\left(1+\frac{1}{2\sqrt{\ln\lambda}}\right) \frac{R}{\eps}}.
\end{eqnarray}
\end{thm}
\begin{rem}
 Note that $j_{n,1}/y_{n,1} = 1+ O(n^{-2/3})$ so for any $\lambda>1$, $n_0$ exists and is finite. Since $j_{n,1}/n = 1+ O(n^{-2/3})<4$, 
the lower bound \eqref{eq:lowerbound-blowup-1} also matches the end of the perturbative regime described in Theorem~\ref{thm:os-lb}, which required
$\oeps< n/\lambda$. 

As we noted earlier, the lower bound blows-up geometrically with $n$. Thus, taking $p=\infty$, if the coefficients $a_n$ decay only polynomially with $n$, then
$$
\sup_{\omega>0} \left\Vert u_{\eps}^{s} \left(|x|=\eps\right) \right\Vert _{H^{\sigma}}=\infty
$$
for any $s>-\infty$. This is the case for plane waves, for example, since
$$
\exp(i\omega x\cdot \zeta) = \sum_{-\infty}^{\infty} \BJ{n}{\omega|\zeta| |x|} 
\exp( i n ( \arg\left(\frac{x}{|x|}- \frac{\zeta}{|\zeta|}\right) +\frac{\pi}{2})).
$$
Estimate \eqref{eq:lowerbound-blowup-0} shows that even low frequencies are affected by quasi-resonances. However, the blow-up is milder. 
Indeed, $\BJ{0}{x}$ tends to $1$ close to the origin, whereas
$$
\left|\BH{0}{\frac{\sqrt{2}}{\lambda\sqrt{\ln \lambda}}\left(1+\frac{1}{2\sqrt{\ln\lambda}}\right)}\right|\approx \frac{2}{\pi} \ln(\lambda). 
$$
This quasi-resonance also occurs just after the perturbative regime, which applies when $\oeps <m_\lambda$. 
We may therefore argue that the estimates provided by Theorem~\ref{thm:ob-ls} are optimal in terms of frequency range, up to a multiplicative factor of at most $4$.
\end{rem}

\begin{proof}[Proof of Theorem~\ref{thm:ob-lr}]
Starting from formula~\eqref{eq:norm-hs}, we have
\begin{eqnarray*}
\sup_{\oeps \leq j_{p,1}/\lambda}\left\Vert  u_{\eps}^{s}\left(|x|=R\right) \right\Vert _{H^{\sigma}} 
&\geq&
\sup_{\oeps \leq j_{p,1}/\lambda}
\sup_{n_0 \leq n \leq p}   
\sqrt{2\pi} \left| \Rn a_{|n|}\right|
\left(1+\left|n\right|\right)^{\sigma}  \left|\BH{n}{\oeps \frac{R}{\eps}}\right|\\
&\geq&
\sup_{n_0 \leq n \leq p}   
\sqrt{2\pi} \left|a_{|n|}\right|
\left(1+\left|n\right|\right)^{\sigma}  \left|\BH{n}{\omega_{n,1} \frac{R}{\eps}}\right|.
\end{eqnarray*}
Where we used that $|R_n(\omega_{n,1},\lambda)|=1$. Since $x\to\left|\BH{n}{x}\right|$ is decreasing and from Proposition~\ref{pro:Dixon} 
$\omega_{n,1} < j_{n,1}/\lambda$,  we obtain \eqref{eq:lowerbound-blowup-1}.
The second bound \eqref{eq:lowerbound-blowup-0} is proved similarly, using the monotonicity of  $x\to\left|\BH{0}{x}\right|$, 
and Lemma~\ref{lem:muzeroone} below which shows that when $\lambda > \exp(2)$, 
$$
\omega_{0,1} < \frac{\sqrt{2}}{\lambda \sqrt{\ln \lambda}}\left(1+\frac{1}{2\sqrt{\ln\lambda}}\right).
$$
\end{proof}

Let us now turn to the proof of Proposition~\ref{pro:Dixon}. 
From the definition of $R_n$ \eqref{eq:def-RN}, it is clear that $R_n(x,\lambda)=-1$ if and only if 
$$
\Im\left(\displaystyle  
 \BHp{n}{x}\BJ{n}{\lambda x}
-\lambda\BJp{n}{\lambda x }\BH{n}{x}\right) = \BYp{n}{x} \BJ{n}{\lambda x} - \lambda \BJp{n}{\lambda x }\BY{n}{x}=0
$$
When $0<x<y_{n,1}$, inequalities \eqref{eq:Dixon-JnYn} show that $\BYp{n}{x}<0$ and $\BYp{n}{x}>0$. 
Dixon's Theorem \cite[15.23]{WATSON-25} shows that $\BJ{n}{x}$ and $\BJp{n}{x}$ have no common zeros. 
Thus quasi-resonances cannot occur at any $j_{n,k}/\lambda$ or $j_{n,k}^{(1)}/\lambda$. 
Lastly, note that when $n>1$,  in the set $(0, j_{n,1}^{(1)}/\lambda)$, both  $\BJ{n}{\lambda x}$ and $ \BJp{n}{\lambda x }$ are positive, 
thus no quasi-resonance can occur. 
The quasi-resonances can only be in the sets $\left(\frac{j_{n,k}^{(1)}}{\lambda} \frac{j_{n,k}}{\lambda}\right)$, and are the solutions of 
\begin{equation}\label{eq:eq-reson-1}
\lambda \frac{\displaystyle \BJp{n}{\lambda x }}{\displaystyle \BJ{n}{\lambda x }} =  
\frac{\displaystyle \BYp{n}{ x }}{\displaystyle \BY{n}{ x }}. 
\end{equation}
Figure~\ref{fig:qr-1} shows a plot of $x\to 2 \BJp{30}{2 x }/ \BJ{30}{2 x }$, in blue, and $x\to \BYp{n}{ x }/\displaystyle \BY{n}{ x }$, 
in red, in the interval $(j_{30,1}^{(1)}/2,y_{30,1})\approx (16.28,32,98)$. The dashed lines represent the solutions of $\BJ{30}{2x}=0$ in this interval.
The eight red dots on the horizontal axis mark the quasi-resonant frequencies corresponding to $n=30$ and $\lambda=2$.  
\begin{figure}
 \begin{center}
  \includegraphics[width=10cm]{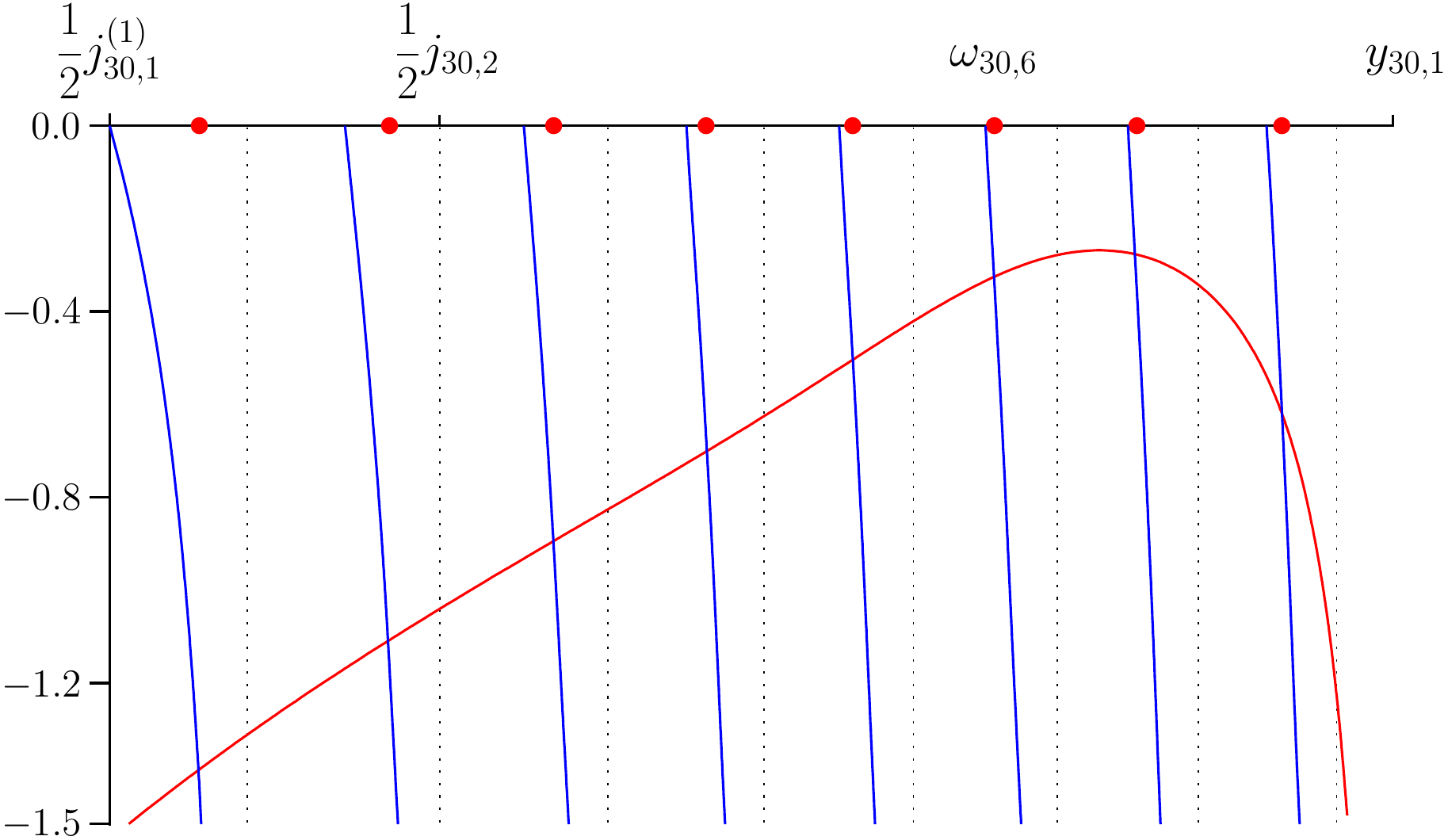}
 \end{center}
\caption{\label{fig:qr-1}Quasi-resonant frequencies for $\lambda=2$ and $n=30$.}
\end{figure}
To study the solutions of \eqref{eq:eq-reson-1}, we introduce, when $n\geq1$
\begin{equation}\label{eq:defgnkn}
g_n:=x\to \frac{x}{n} \frac{\dsp \BJp{n}{x}}{\dsp \BJ{n}{x}} \mbox{ and }k_n:=x\to-\frac{x}{n} \frac{\dsp \BYp{n}{x}}{\dsp \BY{n}{x}},
\end{equation}
 when $n=0$,
\begin{equation}\label{eq:defg0k0}
g_{0}(x):= x\to x \frac{J_{0}^{\prime}(x)}{J_{0}(x)}
\mbox{ and }k_{0}(x):=x\to -x \frac{Y_{0}^{\prime}(x)}{Y_{0}(x)}.
\end{equation}
and we rewrite \eqref{eq:eq-reson-1} as
\begin{equation}\label{eq:eq-reson-2}
 g_n(\lambda x) = - k_n(x).
\end{equation}
\begin{proof}[Proof of Proposition~\ref{pro:Dixon}]
From the recurrence relation satisfied by Bessel functions, we derive that for all $n\geq0$,
\begin{equation}\label{eq:difeqgn}
g_{n}^{\prime}(x)=\frac{n}{x}-\frac{x}{n_{+}}-  \frac{n_{+}}{x}g_{n}^{2}(x),
\end{equation}
with the notation $n_{+}=\max(n,1)$.
In particular, for $x>n$, $g_n$ is decreasing. Thus, on each interval $U_{n,k}$, $x\to g_n(\lambda x)$ decreases from $0$ to $-\infty$. When $n\geq 1$, 
on $[j_{n,1}^{(1)}/\lambda,y_{n,1})$, $k_n$ is positive, thus there exists at least one solution to \eqref{eq:eq-reson-2}. 
When $n=0$, since $(k_0(x)+g_0(x))/x$ tends to $\infty$ as $x$ tends to zero, therefore at least one solution exists in $(0,j_{0,1}/\lambda)$. 

To show uniqueness, we compute that the derivative of $g_n (\lambda\cdot) +k_n(\cdot)$ is
$$
\lambda g_n ^\prime (\lambda x) + k_n^\prime (x) = \frac{x}{n_{+}} \left(1-\lambda^2\right) + \frac{n_{+}}{x}\left(k_n^2(x) - g_n^2(\lambda x)\right),
$$
so at any point where $g_n(\lambda x) = -k_n(x)$, we have 
$$
\lambda g_n ^\prime (\lambda x) + k_n^\prime (x) = \frac{x}{n_{+}} \left(1-\lambda^2\right) < 0.
$$
Thanks to the Intermediate Value Theorem, there can therefore only be one solution. 
Finally, note that when $\lambda<j_{n,1}/y_{n,1}$, $x^{-1}(k_n(x)+g_n(\lambda x))$ tends to $+\infty$ both at $x=0$ and $x=y_{n,1}$, therefore 
there cannot be a unique solution of $k_n(x))+g_n(\lambda x)=0$ in this interval, and consequently there is none.
\end{proof}  

We conclude this section by an upper and lower estimate of $\omega_{0,1}$.

\begin{lem}\label{lem:muzeroone}
The quasi-resonant triplet $(0, \omega_{0,1},\lambda)$ satisfies
\begin{equation}\label{eq:bdmu0one}
\frac{\sqrt{2}}{\lambda \sqrt{\ln \lambda}}\left(1-\frac{1}{2\sqrt{\ln(\lambda)}}\right)
< \omega_{0,1} < 
\frac{\sqrt{2}}{\lambda \sqrt{\ln \lambda}}\left(1+\frac{1}{2\sqrt{\ln\lambda}}\right),
\end{equation} 
for all $\lambda\geq \exp(2)$.
\end{lem}
\begin{proof}
Introducing the functions $f_{+}$ and $f_{-}$ given by
$$
f_\pm(\lambda):= \frac{\sqrt{2}}{\lambda\sqrt{\ln\lambda}}\left(1\pm\frac{1}{2\sqrt{\ln(\lambda)}}\right),
$$
thanks to Proposition~\ref{pro:Dixon}, it is  sufficient to  check that 
$$
g_0(\lambda f_{+}(\lambda))+k_0 ( f_{+}(\lambda)) <0, \mbox{ and } g_0(\lambda f_{-}(\lambda))+k_0 ( f_{-}(\lambda)) >0.
$$
The following bounds
$$
-\frac{1}{2}x^2 - \frac{1}{12} x^4 \leq g_0(x) \leq -\frac{1}{2} x^2 -\frac{1}{16} x^4 \mbox{ for all } 0\leq x \leq 1,
$$
and 
$$
-\frac{1}{\gamma +\ln\left(\frac{x}{2}\right)} +\frac{1}{2}x^2
 \leq k_0(x) \leq 
 -\frac{1}{\gamma +\ln\left(\frac{x}{2}\right)} + x^2 \mbox{ for all } 0\leq x\leq \frac{1}{4}.
$$
can be derived using the asymptotic expansions of Bessel functions around $x=0$ given in \cite{NIST-10}. The proof 
\eqref{eq:bdmu0one} becomes a study of a function of one variable, $\lambda$. 
We omit this tedious but straightforward calculation. It is easy to visually confirm this result using a modern 
scientific computation software, using the built-in formulae of $\BJ{0}{x}$, $\BJ{1}{x}$, $\BY{0}{x}$ and $\BY{1}{x}$ 
to compute $g_0$ and $k_0$, and then verify for example that  $\pm g_0(\lambda f_{\pm}(\lambda))+\pm k_0 ( f_{\pm}(\lambda)) <0$, 
and that both expressions are of order $(\ln \lambda)^{-3/2}$ for large $\lambda$. The lower bound $\lambda\geq \exp(2)$ is not 
optimal: it is convenient because of the form of the ansatz for $k_0$ given above. 
Numerically, it appears that \eqref{eq:bdmu0one} holds almost up threshold value $j_{0,1}/y_{0,1}$ (up to $1.003$ times that value).
\end{proof}

\section{\label{sec:ffield}Inclusions with relative index larger than one: far-field estimates}

As we could notice on Figure~\ref{fig:qr-2}, the effect of quasi-resonances is localized close to $|x|=\eps$. 
We now show that in the far field, that is, when $\lambda \eps \leq|x|$, estimates valid for all frequencies can 
be derived in the spirit of Theorem~\ref{thm:ob-ls}.

\begin{thm}\label{thm:ob-lb}
When  $1\leq\lambda$ and $\eps \lambda <R$ there holds
$$
\sup_{\omega>0} \left\Vert u_{\eps}^{s} \left(|x|=R\right) \right\Vert _{H^{\sigma}} 
\leq 
\frac{5}{2} \sqrt{\lambda \frac{\eps}{R}} \left\Vert u^{i} \left(|x|=\eps\right) \right\Vert _{H^{\sigma}_{*}} 
 + 2\sqrt{\sum_{n\neq 0 } \left|a_{n}\right|^2 \frac{\left(1+\left|n\right|\right)^{2\sigma}}{\left(j_{n,1}^{(1)} \frac{R}{\eps}\right)^2 -n^{2}}}
$$
\begin{equation}\label{eq:bd-up-oblb-1}
 + \sqrt{10\pi}\left|u ^{i}(0)\right| \left|\BH{0}{\min\left(\frac{1}{2},m_{\lambda}\right) \frac{R}{\eps}} \right|,
\end{equation}
Furthermore,
\begin{equation}\label{eq:bd-down-oblb-2}
\sup\limits_{\omega>0} \left\Vert u_{\eps}^{s} \left(|x|=R\right) \right\Vert _{H^{\sigma}_{*}} 
\geq
 \frac{2}{5} \sqrt{\lambda\frac{\eps}{R}} \mathbf{N}_{1}^{\sigma-\frac{1}{6}}\left(u^i\right).
\end{equation}
\end{thm}
\begin{rem}
Just like in Theorem~\ref{thm:ob-ls}, the upper bound \eqref{eq:bd-up-oblb-1} can be replaced by the frequency independent bound
\begin{equation}\label{eq:bd-up-oblb-3}
\left\Vert u_{\eps}^{s} \left(|x|=R\right) \right\Vert _{H^{\sigma}} 
\leq 
 \frac{5}{2} \sqrt{\lambda \frac{\eps}{R}} \mathcal{N}^{\sigma}\left( u^{i} \right) 
 +\sqrt{10\pi}\left|u ^{i}(0)\right| \left|\BH{0}{\min\left(\frac{1}{2},m_{\lambda}\right) \frac{R}{\eps}} \right|.
\end{equation}
We chose the form \eqref{eq:bd-up-oblb-1} to obtain an optimal decay rate in $n$ for $R\gg \eps$.  The dependence on the zero order term is
sharp, as we have seen in Section~\ref{sec:quasi-res}.
Note that the lower bound \eqref{eq:bd-down-oblb-2}  and upper bound \eqref{eq:bd-up-oblb-1} have the same dependence on the contrast, and on $\eps/R$.
\end{rem}
\begin{proof}
Starting from the formula \eqref{eq:norm-hs}, using the monotonicity 
of $x \left|\frac{\BH{n}{\sqrt{q_{0}}\omega R}}{\BH{n}{\lambda\oeps}}\right|^{2}$ 
for $n\geq1$ as in the proof of Theorem~\ref{thm:os-ls}, we obtain
\begin{align*}
\left\Vert  u_{\eps}^{s}\left(|x|=R\right) \right\Vert^2 _{H^{\sigma}} 
&\leq&   
\lambda\frac{\eps}{R}\left( 2\pi \sum_{I_1} \left|a_{n}\right|^2 \left(1+\left|n\right|\right)^{2\sigma} 
 \left|\Rn |\BH{n}{\lambda \oeps}\right|^2   \right)\\
&&  2\pi \sum_{I_2} \left|a_{n}\right|^2 \left(1+\left|n\right|\right)^{2\sigma} 
 \left|\Rn |\BH{n}{\lambda \oeps \frac{R}{\eps}}\right|^2  \\
&&+ 2\pi \left|a_{0}\right|^2 \left|\Ro\BH{0}{\sqrt{q_{0}}\omega R}\right|^{2},
\end{align*}
where $I_1$ is the set of indices $n$ for which $0<|n|$ and $\lambda \oeps < j_{|n|,1}^{(1)}$, and
$I_2$ is the set of indices $n$ for which $0<|n|$ and  $\lambda \oeps > j_{n,1}^{(1)}$. Thanks to Proposition~\ref{pro:lgeq1}, 
$$
\left|\Ro\BH{0}{\sqrt{q_{0}}\omega R}\right| \leq \sqrt{5}\left|\BH{0}{y_{0,1} \frac{R}{\eps}}\right|.
$$
and when $n\in I_1$,
$$
 \left|\Rn \BH{n}{\oeps}\right| \leq \frac{5}{2} \left|\BJ{n}{\oeps}\right|.
$$
Alternatively, as in Theorem~\ref{thm:ob-ls}, when $n\in I_2$,
$$
\left|\Rn \BH{n}{\oeps}\right| \leq\sqrt{\frac{2}{\pi}\frac{1}{\sqrt{\left(j_{n,1}^{(1)} \frac{R}{\eps}\right)^2 -n^{2}}}}
$$ 
Therefore, altogether, we have obtained that for all $\oeps>0$, 
\begin{eqnarray*}
&& 2\pi \sum_{|n|\neq 0} \left|a_{n}\right|^2 \left(1+\left|n\right|\right)^{2\sigma} 
 \left|\Rn |\BH{n}{\oeps}\right|^2  \\
&\leq& \frac{25}{4}  2\pi \sum_{n\neq0} \left|a_{n}\right|^2 \left(1+\left|n\right|\right)^{2\sigma} \left|\BJ{n}{\oeps}\right|^2\\
&& +  4 \sum_{n\neq0} \left|a_{n}\right|^2 \frac{\left(1+\left|n\right|\right)^{2\sigma}}{\left(j_{n,1}^{(1)} \frac{R}{\eps}\right)^2 -n^{2}},
\end{eqnarray*}
which concludes our proof of the upper bound. Turning to the lower bound, we have 
\begin{eqnarray*}
\left\Vert u_{\eps}^{s} \left(|x|=R\right) \right\Vert _{H^{\sigma}} 
 &\geq & \sqrt{2\pi} \sup_{|n|\geq1} \sup_{\oeps>0} 
         \left|R_{n}^{\eps} a_{n}\right| \left(1+\left|n\right|\right)^{\sigma}
         \left|\BH{n}{\oeps \frac{R}{\eps}}\right| \\
& \geq & \sqrt{2\pi} \sup_{|n|\geq 1} |a_{n}| \left(1+\left|n\right|\right)^{\sigma}
        \left|\BH{n}{\omega_{n,1} \frac{R}{\eps}}\right|,
\end{eqnarray*}
where we used the first quasi-resonant frequency $\omega_{n,1}$ given by Proposition~\ref{pro:Dixon}. 
Since $\sqrt{x}\left|H_{n}^{1}(x)\right|$ decreases to $\sqrt{2/\pi}$ \cite[13.74]{WATSON-25}, 
\[
\left|\BH{n}{\omega_{n,1} \frac{R}{\eps}}\right| > \sqrt{\frac{2 \eps}{\pi R \omega_{n,1}}}  
\]
and, in turn, since from \eqref{eq:ineq-landau},
$$
\omega_{n,1}<\frac{j_{n,1}}{\lambda} < \frac{4 n}{\lambda} < 4\frac{(n+1)^{1/3}}{\lambda} \left(\frac{6}{7}\right)^{2}
\left(\sup_{x>0} \left| \BJ{n}{x}\right| \right)^{-2},
$$
we obtain
\[
\left\Vert u_{\eps}^{s} \left(|x|=R\right) \right\Vert _{H_{*}^{\sigma}}  \geq \frac{2}{5} \sqrt{ \lambda \frac{\eps}{R}} 
\mathbf{N}_1^{\sigma-\frac{1}{6}}\left(u^{i}\right),  
\]
as announced.
\end{proof}

\section{\label{sec:bband} Broadband contrast and frequency independent estimates}

In the previous sections, we have seen that we cannot hope for contrast independent estimates for all frequencies, because of the appearance of quasi-resonances. 
Combining Theorem~\ref{thm:ob-ls} and Theorem~\ref{thm:ob-lb}, we see that scattered field tends to zero at a rate $\eps^{1-\delta}$, 
when observed at a fixed distance, say $R=1$, provided the contrast $\lambda$ does not grow faster than $\eps^{-\delta}$ with $\delta<1$. 
On the other hand, when $\lambda$ is of size $\eps^{-1}$ or larger,  the lower bound provided by Theorem~\ref{thm:ob-lb} shows that for some 
frequencies, the quasi-resonant frequencies, some components of the scattered field will be of size one or larger.

The following proposition shows that if an interval around quasi-resonant frequencies is excluded, the scattered field can be controlled by the incident field.
It is proved in Section~\ref{sec:case2}.
\begin{prop} \label{pro:Ito-One} For any $0<\tau\leq\frac{1}{4}$, we define 
\begin{equation}\label{eq:def-itau}
I_{n,k}(\tau):=\left\{ x\in U_{n,k} \mbox{ such that }\left|g_{n}(\lambda x)+k_{n}(x)\right|\leq\tau\left|k_{n}(x)\right|\right\},
\end{equation}
If $\lambda>7$, $n\geq1$ and $\oeps \in (0,y_{n,1})\setminus(\cup_{k} I_{n,k}(\tau))$, then, 
\[
\left|\Rn \BH{n}{\oeps} \right|\leq\frac{9}{2\,\tau} \BJ{n}{\oeps}.
\]
we have
$$
(0,y_{n,1})\setminus(\cup_{k} I_{n,k}(\tau)) = (0,y_{n,1}) \setminus \bigcup_{k\in K(\lambda,n)} I_{n,k}(\tau)
$$
where $K(\lambda,n)$ is the set of all positive $k$ such that $j_{n,k}^{(1)}< n \lambda$. Furthermore,
$$
\left|\bigcup_{k\in K(\lambda,n)} I_{n,k}(\tau) \right| \leq   6 \tau \frac{ n \ln \lambda}{\lambda}.
$$
If $\lambda>7$, $n=0$ and $\oeps \in (0,\zeta_0)\setminus(\cup_{k} I_{0,k}(\tau))$,  where $\zeta_0\approx0.3135$ is defined Proposition~\ref{pro:logYn}, then
\[
\left|\Ro \BH{0}{\oeps} \right|\leq  \frac{5}{3\tau} \BJ{0}{\oeps} \leq \frac{5}{3\,\tau}.
\]
We have
$$
(0,\zeta_0)\setminus(\cup_{k} I_{0,k}(\tau)) = (0,\zeta_0)\setminus \bigcup_{k\in K(\lambda,0)} I_{0,k}(\tau)
$$
where $K(\lambda,0)$ is the set of all positive $k$ such that $j_{0,k}^{(1)}< \lambda \zeta_0$. Furthermore,
$$
\left|\bigcup_{k\in K(\lambda,0)} I_{0,k}(\tau) \right| \leq   7 \tau \frac{ \ln(\ln \lambda)}{\lambda}.
$$
\end{prop}
\begin{proof}
This is the result of Proposition~\ref{pro:Ito} (together with Lemma~\ref{lem:FirstCase} for $\oeps < j_{n,1}^{(1)}/\lambda$ and Proposition~\ref{pro:R0} when $n=0$),
 and Proposition~\ref{pro:Ink}.
\end{proof}

This result allows us to prove the following.

\begin{lem}\label{lem:highcontrast-Lemma}
Suppose $\lambda>7$. Let $\eta_{\max}$ be the following decreasing function of the contrast
\begin{equation}\label{eq:etamax}
\eta_{\max} = \frac{3}{2}  \frac{\ln\lambda}{\lambda}.
\end{equation}
Given $\alpha>0$, for any $\eta>0$ such that
$$
\eta \leq \frac{1}{\alpha} \eta_{\max} 
$$ 
there exists a set $I_1$ depending on $\eta,\alpha,\eps$ and $\lambda$ such that  
$$
|I_1|<\frac{\eta}{\eps}
$$ 
and, for any $R\geq \eps$
$$
\sup\limits_{\sqrt{q_0}\omega \in(0,\infty)\,\setminus \, I_1} \left\Vert  u_{\eps}^{s}\left(|x|=R\right) \right\Vert _{H_{*}^{\sigma}} 
\leq 
18 \sqrt{\frac{\eps}{R}} \frac{\eta_{\max}}{\eta \alpha} \mathcal{N}^{\sigma + 2 + \alpha}\left(u^i\right).
$$
Let $\eta_0$ be given by
$$
\eta_{0} = \frac{7}{4}  \frac{\ln(\ln\lambda)}{\lambda}.
$$
for any $\eta>0$ such that
$$
\eta \leq \eta_{0} 
$$ 
there exists a set $I_{0} \subset(0, \zeta_0)$ which depends on $\eta,\eps$ and $\lambda$ such that  
$$
|I_{0}|<\frac{\eta}{\eps}
$$ 
and, for any $R\geq \eps$
$$
\sup\limits_{\sqrt{q_0}\omega \in(0,\infty)\,\setminus \, I_{0}} \left|\frac{1}{2\pi}\int\limits_0^{2\pi} u_{\eps}^s\left(|x|=R \right)\right|  \leq   
7 \frac{\eta_0}{\eta} \left|\frac{\BH{0}{m_\lambda \frac{R}{\eps}}}{\BH{0}{m_\lambda}}\right|.
$$
\end{lem}
\begin{rem}
Lemma~\ref{lem:highcontrast-Lemma} shows that by excluding some frequencies, around quasi-resonances, near-field estimates can be obtained, 
up to the boundary of the inclusion $|x|=\eps$, at the cost of a little more than two powers of $n$
when compared to the near field estimates given by Theorem~\ref{thm:ob-ls} for $\lambda<1$. 
We showed in Theorem~\ref{thm:ob-lr} that if the quasi-resonances are not excluded, the blow up is geometric in $n$. 
The most striking feature of this result  is that since $\eta_{\max}$ and $\eta_0$  tend to zero as $\lambda$ tends to $\infty$, 
the size of the set of frequencies to exclude shrinks as the contrast increases.
\end{rem}
\begin{proof}[Proof of Lemma~\ref{lem:highcontrast-Lemma}]

Starting from the formula \eqref{eq:norm-hs}, using the monotonicity 
of $x \left|\frac{\BH{n}{\sqrt{q_{0}}\omega R}}{\BH{n}{\oeps}}\right|^{2}$ 
for $n\geq1$ as in the proof of Theorem~\ref{thm:os-ls}, we obtain
$$
\left\Vert  u_{\eps}^{s}\left(|x|=R\right) \right\Vert^2 _{H_{*}^{\sigma}} 
\leq   
\frac{\eps}{R}\left( 2\pi \sum_{|n|\neq 0} \left|a_{n}\right|^2 \frac{1}{\tau_n^2}\left(1+\left|n\right|\right)^{2\sigma} 
 \left|\Rn |\BH{n}{\lambda \oeps}\right|^2   \right),
$$
for any sequence of positive parameters $0< \tau_n < \frac{1}{4}$ to be chosen later. Next we divide the non zero indices into three parts. 
The first set of indices is 
$$
N_1:=\left\{ n\neq0 \mbox{ such that } \lambda \oeps \leq  j_{|n|,1}^{(1)} \mbox{ or } \oeps > n \right\}.
$$ 
In $N_1$, thanks to Proposition~\ref{pro:lgeq1} and Proposition~\ref{pro:ntoyn1}, we have that either
$$
\left|\Rn |\BH{n}{\lambda \oeps}\right| \leq \frac{5}{2} \BJ{n}{\oeps} \leq \frac{5}{2} \sup_{x>0} \left| \BJ{n}{x} \right|
$$ 
or $\oeps>y_{n,1}$, and \eqref{eq:bd-bignx} shows that 
$$
\left|\Rn |\BH{|n|}{\lambda \oeps}\right|^2 \leq \frac{7}{5} \sup_{x>0} \left| \BJ{n}{x} \right|
$$
We define the sets
$$
N_2:=\left\{ n\neq0 \mbox{ such that }  \left|g_{|n|}(\lambda x)+k_{|n|}(x)\right|\leq\tau_n\left|k_{|n|}(x)\right|\right\},
$$
and
$$
N_3:=\left\{ n\neq 0, n\not\in S_1 \mbox{ and } n\not\in S_2\right\}.
$$
Proposition~\ref{pro:Ito-One} shows that for all $n\in N_3$
$$
\left|\Rn \BH{n}{\oeps}\right|  
\leq  \frac{9}{2}\frac{1}{\tau_n}  \BJ{n}{\oeps} \leq \frac{9}{2} \frac{1}{\tau_n} \sup_{x>0} \left| \BJ{n}{x} \right| 
$$
We have obtained that for all $\omega$ such that $N_2=\emptyset$, we have 
\begin{equation}\label{eq:pfbd-tau-1}
\left\Vert  u_{\eps}^{s}\left(|x|=R\right) \right\Vert^2 _{H_{*}^{\sigma}} 
 < 2\pi \left(\frac{9}{2}\right)^2 \frac{\eps}{R}  \sum_{n\neq0} \left|a_{n}\right|\frac{1}{\tau_n^2}\left(1+\left|n\right|\right)^{2\sigma} \sup_{x>0} \left| \BJ{n}{x} \right|^2. 
\end{equation}
Thanks to Proposition~\ref{pro:Ito-One}, the forbidden values for $\sqrt{q_0}\omega$ lie 
in a collection of intervals $O_\eps := \eps^{-1} \bigcup_{n,k} I_{n,k}(\tau_n)$ of total size  of at most
$$
|O_\eps|\leq  \sum_{n=1}^{\infty}  n \tau_n \frac{6 \ln \lambda}{\eps \lambda} = 4\frac{\eta_{\max}}{\eps} \sum_{n=1}^\infty  n \tau_n .
$$
This leads us to choose
$$
\tau_n = \frac{\eta\alpha}{(1+|n|)^{2+\alpha}} \frac{1}{4 \eta_{\max}}.
$$
Then, an upper estimate of the total size of the forbidden intervals is 
$$
|O_\eps|\leq \frac{1}{\eps}\sum_{n=1}^{\infty} \frac{\alpha}{(1+n)^{1+\alpha}} \leq \frac{\eta}{\eps},
$$
and from \eqref{eq:pfbd-tau-1} we obtain
$$
\left\Vert  u_{\eps}^{s}\left(|x|=R\right) \right\Vert_{H_{*}^{\sigma}} 
 \leq  18 \sqrt{\frac{\eps}{R}} \frac{\eta_{\max}}{\eta \alpha} \mathcal{N}^{\sigma + 2  + \alpha}\left(u^i\right),
$$
as announced.

Let us now consider the other estimate. We have, from \eqref{eq:usca}
$$
\left|\frac{1}{2\pi}\int\limits_0^{2\pi} u_{\eps}^s\left(|x|=R \right)\right| =  |u^i(0)| \left| \Ro \BH{0}{\oeps \frac{R}{\eps}}\right|. 
$$
Proposition~\ref{pro:R0plus} shows that when $\oeps < m_\lambda$,
$$
\left| \Ro \BH{0}{\oeps \frac{R}{\eps}}\right| \leq 4 \left|\frac{\BH{0}{m_\lambda \frac{R}{\eps}}}{\BH{0}{m_\lambda}}\right|\leq \frac{1}{\tau}\left|\frac{\BH{0}{m_\lambda \frac{R}{\eps}}}{\BH{0}{m_\lambda}}\right|.
$$
Proposition~\ref{pro:R0} shows that when $\oeps > \zeta_0$,
$$
\left| \Ro \BH{0}{\oeps \frac{R}{\eps}}\right| \leq  \left|\BH{0}{\zeta_0 \frac{R}{\eps}}\right| \leq \sqrt{\frac{\eps}{R}} \sqrt{\frac{2}{\pi \zeta_0}} \leq \frac{3}{2} \sqrt{\frac{\eps}{R}}.
$$
From Lemma~\ref{lem:logconcave}, we know that 
$$
\left|\frac{\BH{0}{m_\lambda \frac{R}{\eps}}}{\BH{0}{m_\lambda}}\right| \geq \sqrt{\frac{\eps}{R}},
$$
therefore when $\oeps > \zeta_0$,
$$
\left| \Ro \BH{0}{\oeps \frac{R}{\eps}}\right| \leq \frac{1}{\tau} \left|\frac{\BH{0}{m_\lambda \frac{R}{\eps}}}{\BH{0}{m_\lambda}}\right|.
$$
On the other hand, Proposition~\ref{pro:Ito-One} shows that if $\oeps\in(m_\lambda,\zeta_0)\setminus \bigcup_{k\in K(\lambda,0)} I_{0,k}(\tau)$ we have
$$
\left| \Ro \BH{0}{\oeps}\right| \leq  \frac{5}{3} \frac{1}{\tau}. 
$$
Therefore, Lemma~\ref{lem:logconcave} shows that $x\to \left|{\BH{0}{x\frac{R}{\eps}}}/{\BH{0}{x}}\right|$  is decreasing,
\begin{eqnarray*}
\left| \Ro \BH{0}{\oeps \frac{R}{\eps}}\right| &=& \left| \Ro \BH{0}{\oeps}\right| \left|\frac{\BH{0}{\oeps \frac{R}{\eps}}}{\BH{0}{\oeps}}\right|\\
                                               &\leq& \frac{5}{3}\frac{1}{\tau}\left|\frac{\BH{0}{m_\lambda \frac{R}{\eps}}}{\BH{0}{m_\lambda}}\right|.
\end{eqnarray*}
We have obtained, for all $\oeps \in (0,\infty)\setminus  \bigcup_{k\in K(\lambda,0)} I_{0,k}(\tau)$, and all $\tau \leq \frac{1}{4}$,
$$
\left| \Ro \BH{0}{\oeps \frac{R}{\eps}}\right| \leq  \frac{5}{3}\frac{1}{\tau} \left|\frac{\BH{0}{m_\lambda \frac{R}{\eps}}}{\BH{0}{m_\lambda}}\right| 
$$
The total size of the set of forbidden values for $\sqrt{q_0} \omega$ is bounded by
$$
|O_\eps|\leq  7\frac{\ln \ln(\lambda)}{\lambda} \frac{\tau}{\eps},
$$
thus choosing $\tau =  \frac{1}{4} \frac{\eta}{\eta_0}$ establishes our claim.
\end{proof}
As an application of Lemma~\ref{lem:highcontrast-Lemma}, we provide a broadband estimate for the case $\lambda=\eps^{-1}$.
\begin{cor}\label{cor:broadband}
Assume $\eps<1/7$, and $\lambda=\eps^{-1}$. Then, for any $\alpha,\beta>0$, there exists a set $I_1$ depending on $\eps,\alpha$ and $\beta$ and a set $I_0$ depending
on $\eps$ and $\beta$ which satisfies
$$
|I_1| \leq \eps^{\beta} \lneps,\quad |I_0|\leq  \frac{\ln \lneps}{\left(\lneps+1\right)^{\beta}},
$$
such that for all $R\geq\eps$,
$$
\sup\limits_{\sqrt{q_0}\omega \in(0,\infty)\,\setminus \, I_1} \left\Vert  u_{\eps}^{s}\left(|x|=R\right) \right\Vert^2 _{H_{*}^{\sigma}}  
\leq \frac{18}{\alpha} \sqrt{\frac{\eps^{1-2\beta}}{R}}\mathcal{N}^{\sigma + 2 + \alpha}\left(u^i\right),
$$
and
$$
\sup\limits_{\sqrt{q_0}\omega \in(0,\infty)\,\setminus \, I_0} \left|\frac{1}{2\pi}\int\limits_0^{2\pi} u_{\eps}^s\left(|x|=R \right)\right|  \leq   
12 \frac{1}{\sqrt{\left(\lneps+1\right)^{3/2 -2\beta} R}}. 
$$
\end{cor}
\begin{proof}
This is an application of Lemma~\ref{lem:highcontrast-Lemma}. In this case $\eta_{\max} = \eps \lneps$, and  $\eta_0 = \eps \ln\lneps$.
Choose $\eta = \eps^{\beta} \eta_{\max}$. We have
$$
|I_1|\leq \eps^{\beta} \lneps
$$
and
$$
\sup\limits_{\sqrt{q_0}\omega \in(0,\infty)\,\setminus \, I_1} \left\Vert  u_{\eps}^{s}\left(|x|=R\right) \right\Vert^2 _{H_{*}^{\sigma}}  
\leq \frac{21}{\alpha} \sqrt{\frac{\eps^{1-2\beta}}{R}}.
$$
Choose $\eta=\left(\lneps+1\right)^{-\beta}\eta_0$. Then,
$$
|I_0|\leq  \frac{\ln \lneps}{\left(\lneps+1\right)^{\beta}}.
$$
Using the bound $\left|\BH{0}{m_{\lambda}} \right|> \frac{1}{2} (1 + \lneps)$, and the usual upper bound \eqref{eq:bd-ho-s},
$$
\sup\limits_{\sqrt{q_0}\omega \in(0,\infty)\,\setminus \, I_0} \left|\frac{1}{2\pi}\int\limits_0^{2\pi} u_{\eps}^s\left(|x|=R \right)\right|  \leq   
14 \sqrt{\frac{2}{\pi}} \left(\lneps+1\right)^{\beta - 1 + \frac{1}{4}}\frac{1}{\sqrt{R}} \leq 12 \frac{1}{\sqrt{\left(\lneps+1\right)^{3/2 -2\beta} R}}. 
$$
\end{proof}

Combining Lemma~\ref{lem:highcontrast-Lemma} with Theorem~\ref{thm:ob-ls} and Theorem~\ref{thm:ob-lb} we obtain the following broadband result, 
which provides a uniform estimate for all contrast, and almost all frequencies.

\begin{thm}\label{thm:broadband}
Suppose given $\lambda>0$ and  $\frac{1}{15}>\eps>0$. For any $\alpha>0$, there exists a open set  
$I_1\left(\alpha,\lambda,\eps\right)\subset\mathbb{R}$  satisfying
$$
|I_1\left(\alpha,\lambda,\eps\right)| \leq \eps^{1/8} \lneps
$$
and for any $R\geq \eps^{1/4}$ we have 
\begin{equation}\label{eq:bd-bbd}
\sup\limits_{\sqrt{q_0}\omega \in(0,\infty)\,\setminus \, I_1} \left\Vert  u_{\eps}^{s}\left(|x|=R\right) \right\Vert^2 _{H_{*}^{\sigma}} 
\leq 
 \frac{21}{\alpha} \sqrt{\frac{\eps^{1/4}}{R}}   \mathcal{N}^{\sigma + 2 + \alpha}\left(u^i\right).
\end{equation}
When $\lambda\leq \eps^{-3/4}$, \eqref{eq:bd-bbd} holds with $I_1=\emptyset$. 
There exists another open set $I_0\left(\eps,\lambda\right)\subset(0,\zeta_0)$ satisfying
$$
|I_0\left(\lambda,\eps\right)| \leq \frac{ \ln \ln \eps + 2}{\ln \eps +1} 
$$
and such that for any $R$ such that $\left(\lneps +1\right)^{1/12} R \leq 1 $, we have 
\begin{equation}\label{eq:bd-bbd0}
\sup\limits_{\sqrt{q_0}\omega \in(0,\infty)\,\setminus \, I_0} \left|\frac{1}{2\pi}\int\limits_0^{2\pi} u_{\eps}^s\left(|x|=R \right)\right|  \leq   
21 \max\left( \sqrt{\frac{2}{\left(\lneps +1\right)^{1/12} R}}, \left|\frac{\BH{0}{m_\lambda \frac{R}{\eps}}}{\BH{0}{m_\lambda}}\right|\right),
\end{equation}
where $m_\lambda$ is defined in \eqref{eq:mlambda} and equals 
$$
m_\lambda = \frac{1}{ \lambda \sqrt{\ln \lambda +1 }}.
$$
When $\lambda^{-1}> \eps\left(\lneps +1\right)^{7/12} $, \eqref{eq:bd-bbd0} holds with $I_0=\emptyset$.
\end{thm}
\begin{rem}
Note that both $I_1$ and $I_0$ have zero measure when $\eps$ tends to zero: in this limit, the estimate is true almost everywhere. 
This is not a far-field result in the sense of \ref{thm:ob-lb}, as we do not require $\lambda R > \eps$.  
Any fixed positive $R$ is possible for $\eps$ small enough.
The decay of the size of $I_0$ is logarithmically slow, even though the quasi-resonances of the zero order term 
are at most logarithmic in the contrast, see Theorem~\ref{thm:ob-lr}. Note that the upper bound in \eqref{eq:bd-bbd0} is 
always smaller than the numerical constant $21$. 
Naturally, many variants of Theorem~\ref{thm:broadband} can be derived by other combinations of Lemma~\ref{lem:highcontrast-Lemma}, 
Theorem~\ref{thm:ob-ls} and \ref{thm:ob-lb}, and more precise broadband estimates can be obtained if $\lambda$ is known to be in a 
particular range  with respect to $\eps$. Corollary~\ref{cor:broadband} gives an improved estimate in the case $\lambda=\eps^{-1}$.
The dependence on $n$ in \eqref{eq:bd-bbd} could probably be improved by a  precise study of the distance between the roots
of the Bessel Function $\BJ{n}{x}$ for $x\in(n/\eps,n)$.
\end{rem}

\begin{proof}
When $0<\lambda \leq 1$, Theorem~\ref{thm:ob-ls} implies \eqref{eq:bd-bbd}, with $I_1=\emptyset$. 
Similarly, when  $1<\lambda<\eps^{-3/4}$,  Theorem~\ref{thm:ob-lb} implies \eqref{eq:bd-bbd}, with $I_1=\emptyset$. .
Let us therefore suppose $\lambda>\eps^{-3/4}>7$. 
We apply Lemma~\ref{lem:highcontrast-Lemma} with
$$
\eta= \frac{8}{9} \varepsilon^{3/8} \eta_{\max} . 
$$
Note that Since $\lambda \to \lambda^{-1} \ln \lambda$ is decreasing when $\lambda>7$, therefore
$$
\eta_{\max} \leq \frac{9}{8}  \lneps \eps^{3/4},
$$ 
and there exists a set $I_1$ such that
$$
|I_1| \leq  \eps^{1/8} \lneps 
$$
for which
$$
\sup\limits_{\sqrt{q_0}\omega \in(0,\infty)\,\setminus \, I_1} \left\Vert  u_{\eps}^{s}\left(|x|=R\right) \right\Vert^2 _{H_{*}^{\sigma}} 
\leq 
 \frac{21}{\alpha} \sqrt{\frac{\eps^{1/4}}{R}}   \mathcal{N}^{\sigma + 2 + \alpha}\left(u^i\right).
$$

\smallskip{}

Let us now turn to the zero order term, and assume $u^{i}(0)=1$ by linearity. 
The cases $\lambda<1$ or $m_\lambda > 1/2$ are consequences of Theorem~\ref{thm:ob-ls} and Theorem~\ref{thm:ob-lb} with $I_0=\emptyset$.
When $m_\lambda<\frac{1}{2}$, and
$$
\lambda\leq \lambda_0:=\frac{1}{\eps\left(\lneps +1\right)^{7/12}},
$$
Theorem~\ref{thm:ob-lb} shows that with $I_0=\emptyset$,
$$
\left|\frac{1}{2\pi}\int\limits_0^{2\pi} u_{\eps}^s\left(|x|=R \right)\right|\leq \sqrt{10\pi} 
\left|\BH{0}{m_{\lambda}  \frac{R}{\eps}} \right|.
$$ 
Since $\BH{0}{\cdot}$ is decreasing, an upper bound is found by choosing $\lambda = \lambda_0$. Then, $\ln \lambda_0 \leq \lneps$, and 
$$
\eps^{-1} m_\lambda = \left(\ln \eps +1\right)^{7/12}\frac{1}{\sqrt{\ln \lambda_0 +1}} \geq \left(\lneps +1\right)^{1/12}
$$
which yields
\begin{equation}\label{eq:bbd-01}
\left|\frac{1}{2\pi}\int\limits_0^{2\pi} u_{\eps}^s\left(|x|=R \right)\right| \leq \sqrt{10\pi}\left|\BH{0}{m_{\lambda}  \frac{R}{\eps}} \right| \leq 5 \frac{1}{\sqrt{R\left(\lneps +1\right)^{1/12}}}.
\end{equation}
When 
$$ 
\lambda_0 <\lambda < \frac{(\lneps +1)^{1/12}}{\eps},
$$
we have
$$
m_\lambda > m_1:=\frac{1}{2+e^{-2}} \frac{\eps}{\left(1+\lneps \right)^{7/12}}.
$$
We turn now to Lemma~\ref{lem:highcontrast-Lemma}. We have
$$
\eta_0 = \frac{7}{4} \frac{\ln \ln \lambda}{\lambda} \leq \frac{\ln \ln \lambda_0}{\lambda_0} \leq \eps \ln (\lneps) \left( 1 + \lneps \right)^{7/12}
$$ 
Choose 
$$
\eta =  \frac{4}{7} \frac{1}{(1+ \lneps)^{2/3}} \eta_0.
$$
Then,
$$
\left| I_0 \right| \leq  \frac{ \ln\lneps}{(\lneps +1)^{1/12}}
$$
and since $x\to\left|\BH{0}{x \frac{R}{\eps}} \right|/ \left|\BH{0}{x}\right|$ is decreasing by Lemma~\ref{lem:logconcave}, outside of $I_0$ we have
\begin{equation}\label{eq:bbd-02}
\left|\frac{1}{2\pi}\int\limits_0^{2\pi} u_{\eps}^s\left(|x|=R \right)\right|\leq  7 \left(\frac{7}{4} (1+ \lneps)^{2/3}\right)
\frac{\left|\BH{0}{m_{1}  \frac{R}{\eps}} \right|}{\left|\BH{0}{m_{1}} \right|}.
\end{equation}
We have $\left|\BH{0}{m_{1}} \right|> \frac{1}{2} (1 + \lneps)$, and therefore
$$
\left|\frac{1}{2\pi}\int\limits_0^{2\pi} u_{\eps}^s\left(|x|=R \right)\right|
\leq \frac{49}{2}\sqrt{\frac{2(2+e^{-2})}{\pi}} (1+ \lneps)^{2/3 -1 + 7/14}\frac{1}{\sqrt{R}}.
\leq 29 \frac{1}{\sqrt{R\left( 1 + \lneps \right)^{1/12}}}.
$$
Let us now assume $\lambda \geq \frac{(\lneps +1)^{1/12}}{\eps}$. In that case, 
$$
\eta_0 =\frac{7}{4} \frac{ \ln \left( \ln \lambda\right)}{\lambda} \leq \frac{11}{6} \eps \frac{\ln \lneps}{\left(1 + \lneps\right)^{1/12}}.
$$
Choosing $\eta=\frac{6}{11} \eta_0$ and applying Lemma~\ref{lem:highcontrast-Lemma}, we obtain
\begin{equation}\label{eq:bbd-03}
|I_0|\leq \frac{\ln \lneps}{\left(1+\lneps\right)^{1/12}},
\, \mbox{ and } \,
\sup_{\sqrt{q_0}\omega \in (0\infty)\setminus I_0} \left|\frac{1}{2\pi}\int\limits_0^{2\pi} u_{\eps}^s\left(|x|=R \right)\right|\ \leq 21 \frac{\left|\BH{0}{m_{\lambda}  \frac{R}{\eps}} \right|}{\left|\BH{0}{m_{\lambda}} \right|}.
\end{equation}
Combining \eqref{eq:bbd-01}, \eqref{eq:bbd-02} and \eqref{eq:bbd-03} we obtain the \eqref{eq:bd-bbd0}.
\end{proof}

\section{\label{sec:caseone} Estimates relating the scattered field and the incident field outside quasi-resonances}
\subsection{The $n\neq0$ case}
The main result of this section, Lemma~\ref{lem:FirstCase} proves Proposition~\ref{pro:lleq1} and Proposition~\ref{pro:lgeq1} when $n\neq0$.
Note that, for a given $n$, these results are focused on the case when $\oeps$ is bounded, namely $\oeps<y_{n,1}$.
Since $y_{n,1}<j_{n,1}$ for all $n\geq0$, we are thus only considering the case when $\BJ{n}{\oeps}>0$, and $\oeps<y_{n,1}$.
To compare the scattered field with the incident field, it is convenient to introduce a new quantity, namely
\begin{equation}\label{eq:def-sn}
S_n(\oeps) := - \frac{\Rn \BH{n}{\oeps}}{\BJ{n}{\oeps}}. 
\end{equation}
A simple manipulation of the equation, together with the Wronskian identity satisfied by $\BJ{n}{x}$ and $\BY{n}{x}$ 
shows that $S_n$ has two equivalent formulations, namely
\begin{eqnarray}
{S_n}(x) &=&\frac{\left(g_{n}(\lambda x)-g_{n}(x)\right)\left(1+i\tan\theta_{n}(x)\right)}{\left(g_{n}(\lambda x)-g_{n}(x)\right)
+i\tan\theta_{n}(x)\left(g_{n}(\lambda x)+k_{n}(x)\right)},\label{eq:formula-Sn-1}\\
    &=& \frac{u_n(x) \BH{n}{x}}{ u_n(x) \BH{n}{x} + i \frac{2}{\pi} \BJ{n}{\lambda x} },\label{eq:formula-Sn-2}\\
\end{eqnarray}
where $\theta_{n}$ is given by \eqref{eq:def-MnThetan},  $g_n$ and $k_n$ are defined by \eqref{eq:defgnkn} and \eqref{eq:defg0k0} and where $u_n$ is given by 
$$
u_n(x) := \max(|n|,1) \BJ{n}{x} \BJ{n}{\lambda x} \left( g_{n}(x) - g_{n}(\lambda x)\right).
$$
Note that these formulae are not defined when $\lambda x = j_{n,1}$, that is, at the singular points of  $g_n (\lambda x)$. However from 
Formula~\eqref{eq:formula-Sn-2} we see that these singularity can be resolved, and we define $S_n$ as the continuous limit of $S_n$ 
towards these points (which is $1$).

The main result of this section is the following Lemma.
\begin{lem} \label{lem:FirstCase} For all $\lambda\in(0,\infty)$, and all $n\geq1$,
for $x\in\left[0,\min\left(j_{n,1}^{(1)}\lambda^{-1},y_{n,1}\right)\right]$
we have the following bound 
\[
\left|{S_n}(x)\right|\leq\frac{5}{2}.
\]
For  $x\in\left[0,\min\left(n\lambda^{-1},n\right)\right]$, we also have
\[
\left|{S_n}(x)\right|\leq 2 \left| \lambda - 1 \right| \frac{x}{n^{1/3} }.
\]
and if $\lambda^{2}<1-\left(\frac{7}{3 n^{1/3}}\right)^{2}$, then 
$$
\left|{S_n}(n)\right|>\frac{1}{2}.
$$
\end{lem}
To prove Lemma~\ref{lem:FirstCase}, we shall use the following observation.
\begin{prop}\label{pro:estimate5half}
For all $0<x\leq n$, and all $n\geq1$,
we have the following bounds \[
\left(\frac{3}{5}\right)^{2}<\frac{1+Y_{n}^{2}(x)/J_{n}^{2}(x)}{1+\left|Y_{n}^{\prime}(x)\right|^{2}/\left|J_{n}^{\prime}(x)\right|^{2}}\leq\left(\frac{5}{2}\right)^{2}.\]
\end{prop}

We do not use the lower bound in this paper. We include it to document the fact that one cannot hope for an upper bound tending to zero for $n$ large, for example.

\begin{proof} The proof is elementary from the inequalities \eqref{eq:boundknovergn} given in Appendix~\ref{ap:A}.
Since \[
\frac{2}{5}<\alpha_{n}=\frac{k_{n}}{g_{n}}=\left|\frac{Y_{n}^{\prime}}{J_{n}^{\prime}}\frac{J_{n}}{Y_{n}}\right|<\frac{5}{3},\]
it suffices to observe that\[
\frac{1+Y_{n}^{2}(x)/J_{n}^{2}(x)}{1+\left|Y_{n}^{\prime}(x)\right|^{2}/\left|J_{n}^{\prime}(x)\right|^{2}}=\frac{1+\tan^{2}\theta_{n}}{1+\alpha_{n}^{2}\tan^{2}\theta_{n}}\]
therefore, since $\frac{2}{5}<1<\frac{5}{3}$,
\[
\left(\frac{3}{5}\right)^{2}<\frac{1+Y_{n}^{2}(x)/J_{n}^{2}(x)}{1+\left|Y_{n}^{\prime}(x)\right|^{2}/\left|J_{n}^{\prime}(x)\right|^{2}}<\left(\frac{5}{2}\right)^{2}.
\]
\end{proof}

The final intermediate result we shall use is the following bound.
\begin{prop}
\label{pro:ntoyn1}For all $x\in[n,y_{n,1}]$, we have 
\[
\left|{S_n}\right|<\sqrt{5}.
\]
\end{prop}

\begin{proof}
Using that $|\Rn|\leq 1$, we have
$$
|{S_n}|^2 \leq 1 + \tan^2 \theta_{n}.
$$
Remember that from \eqref{eq:prothetan}, as $x$ varies between $0$ and $y_{n,1}$, $\theta_{n}(x)$ varies between
$-\frac{\pi}{2}$ and $0$. The map of $n\to-\theta_{n}(n)$ is decreasing to $\pi/3$, and is always close to its limit, as 
\[
\frac{9}{5}>\tan\left(-\theta_{n}(n)\right)>\sqrt{3}\mbox{ for }n\geq1.
\]
see \cite[15.8]{WATSON-25}.
Consequently, for all 
$x\in[n,y_{n,1}]$, $\tan(-\theta_{n}(x))\leq\tan(-\theta_{n}(n))<\frac{9}{5}$, and 
\[
\left|{S_n}\right|^{2}\leq1+\tan(\theta_{n}(x))^{2}<5,
\]
as claimed.
\end{proof}

We can now conclude the proof of the estimate of this section.

\begin{proof}[Proof of Lemma~\ref{lem:FirstCase}]
Let us first consider the case when $x\in\left[0,\min\left(n,j_{n,1}^{(1)}\lambda^{-1},y_{n,1}\right)\right]$. 
When $x>n$, the result follows from Proposition~\ref{pro:ntoyn1}.

We have 
\begin{equation}\label{eq:formtn1}
\left|{S_n}(x)\right|^{2}=\frac{\left|g_{n}(\lambda x)-g_{n}(x)\right|^{2}\left(1+\tan^{2}\theta_{n}\right)}{\left|g_{n}(\lambda x)-g_{n}(x)\right|^{2}+\tan^{2}\theta_{n}\left|g_{n}(\lambda x)+k_{n}(x)\right|^{2}}.
\end{equation}
From Proposition~\ref{pro:propsg}, for $\lambda x\leq j_{n,1}^{(1)}$, $g_{n}(\lambda x)\geq0$, and $k_{n}(x)>0$. The study of the function
\[
u\to\frac{(u-a)^{2}\left(1+\tan^{2}\theta_{n}\right)}{(u-a)^{2}+\tan^{2}\theta_{n}(u+b)^{2}}\]
for $u>0,a>0$ and $b>0$ shows that it has a minimum for $u=a$,
tends to $1$ for $u\to\pm\infty$ and decreases between $0$ and
$a$. Therefore, 
\[
\left|{S_n}\right|^{2}\leq\max\left(1,\frac{\left|g_{n}\right|^{2}\left(1+\tan^{2}\theta_{n}\right)}{\left|g_{n}\right|^{2}+\tan^{2}\theta_{n}\left|k_{n}\right|^{2}}\right).
\]
Now compute that
$$ 
 \frac{\left|g_{n}(x)\right|^{2}\left(1+\tan^{2}(x)\theta_{n}(x)\right)}{\left|g_{n}(x)\right|^{2}+\tan^{2}\theta_{n}(x)\left|k_{n}(x)\right|^{2}}
 =\frac{1+\BY{n}{x}^{2}/\BJ{n}{x}^{2}}{1+\left|\BYp{n}{x}\right|^{2}/\left|\BJp{n}{x}\right|^{2}},
$$
and thanks to the Proposition~\ref{pro:estimate5half} this quotient is bounded by $\left(\frac{5}{2}\right)^{2}$.  
We have obtained that $|S_n|\leq 5/2$.   From  \eqref{eq:formtn1}, we derive that
$$
\left|{S_n}(x)\right|^{2} \leq  \left|\frac{g_{n}(\lambda x)-g_{n}(x)}{k_n(x)+g_n(\lambda x) }\right|^{2} \frac{1+\tan^2\theta_{n}(x)}{\tan^2\theta_{n}(x)}.
$$
As we will see in Proposition~\ref{pro:ntoyn1}, $ \tan(\theta_{n})^{-2} +1 \leq \frac{4}{3}$ when $x\leq n$. 
Thanks to Proposition~\ref{pro:propsg}, $g_n(\lambda x) +k_n(x) \geq k_n(x) > \frac{3}{5} n^{-1/3}$, thus we have
$$
\left|{S_n}(x)\right|^{2} \leq \frac{100}{27} n^{2/3} (\lambda - 1)^2 x^2 \left(g_n^\prime(n)\right)^2 \leq 4 (\lambda - 1)^2 \frac{x^2}{n^{2/3}},
$$
where we used the bounds on $g_n^\prime$ given by Proposition~\ref{pro:propsg}.

Let us now turn to the lower bounds. When $\lambda<1$, Consider the case $x=n$. Then Proposition~\ref{pro:propsg} shows that
$g_{n}(\lambda n)>\sqrt{1-\lambda^{2}}$,
 \[
\frac{7}{6}\frac{1}{n^{1/3}}>k_{n}(n)\mbox{ and } \frac{13}{14n^{1/3}}>g_{n}(n)>\frac{1}{\sqrt{2}}\frac{1}{n^{1/3}}
\]
 and $\left|\tan^{2}\theta_{n}(n)\right|>3$, therefore 
\begin{align*}
\left|{S_n}(x)\right|^{2} 
& =\frac{\left|g_{n}(\lambda n)-g_{n}(n)\right|^{2}\left(1+\tan^{2}\theta_{n}(n)\right)}
{\left|g_{n}(\lambda n)-g_{n}(n)\right|^{2}+\tan^{2}\theta_{n}(n)\left|g_{n}(\lambda n)+k_{n}(n)\right|^{2}}\\
 & >\frac{4\left|g_{n}(\lambda n)-g_{n}(n)\right|^{2}}
     {\left|g_{n}(\lambda n)-g_{n}(n)\right|^{2}+3\left|g_{n}(\lambda n)+\frac{7}{6}n^{-1/3}\right|^{2}}\\
& >\frac{4\left|\sqrt{1-\lambda}-\frac{1}{n^{1/3}}\right|^{2}}
     {\left|\sqrt{1-\lambda}-n^{-1/3}\right|^{2}+3\left|\sqrt{1-\lambda}+\frac{7}{6}n^{-1/3}\right|^{2}}\\
 & >\frac{1}{4},
\end{align*}
when $\sqrt{1-\lambda^{2}}>7/(3n^{1/3}),$ or $\lambda<\sqrt{1-\left(\frac{7}{3n^{1/3}}\right)^{2}}$.
\end{proof}
\subsection{The $n=0$ case}
We summarize here properties of $g_0$ and $k_0$. 
They are derived using methods similar to the ones used for $k_n$ an $g_n$ for $n\neq0$, and can be checked by inspection with the help of a modern mathematical software. 
We therefore will omit the proof.

\begin{prop}\label{pro:propg0k0}
The function $g_0$ (resp. $k_0$) is defined on $(0,\infty)$ except at  $j_{0,k}$ (resp. $y_{0,k}$), $k=1,\ldots$, 
and cancels at each $j_{0,k}^{(1)}$ (resp. $y_{0,k}^{(1)}$). 
\begin{itemize}
 \item Where it is defined, $g_0$ is decreasing. On $(0,j_{0,1})$ $g_0$ is concave. 
$$
\lim_{\genfrac{}{}{0pt}{}{x\to j_{0,k}}{x < j_{0,k}}} g_0(x) = -\infty, \quad  \lim_{\genfrac{}{}{0pt}{}{x\to j_{0,k}}{x > j_{0,k}}} g_0(x) = +\infty.
$$
 \item Where it is defined, $k_0$ is increasing, and
$$
\lim_{\genfrac{}{}{0pt}{}{x\to y_{0,k}}{x < y_{0,k}}} k_0(x) = +\infty, \quad  \lim_{\genfrac{}{}{0pt}{}{x\to y_{0,k}}{x > y_{0,k}}} k_0(x) = -\infty.
$$
\item For $0\leq x \leq 1$,
$$
-\frac{1}{2}x^2 - \frac{1}{12} x^4 \leq g_0(x) \leq -\frac{1}{2} x^2 -\frac{1}{16} x^4, \mbox{ and } |g_0^\prime| \leq \frac{4}{3} x.
$$
\item For $0\leq x\leq\frac{1}{4}$,
$$
-\frac{1}{\gamma +\ln\left(\frac{x}{2}\right)} +\frac{1}{2}x^2
 \leq k_0(x) \leq 
 -\frac{1}{\gamma +\ln\left(\frac{x}{2}\right)} + x^2.
$$
\end{itemize}
\end{prop}

\begin{prop}
\label{pro:n0}
For all $\lambda\leq1$,
and all $0<x< y_{0,1}$, 
\[
\left|{S_0}(x)\right|\leq \frac{\pi}{2\sqrt{2}} \min(1, 2-2\lambda)  x^2  \left|\ln\left(\frac{x}{2}\right)\right|. 
\]
For all $\lambda\geq1$ and all $0<x \leq \min(\frac{1}{2},m_\lambda)$,
\[
\left|{S_0}(x)\right| \leq  \pi\min\left(1,\frac{5}{2} \frac{\lambda-1}{\lambda}\right)  \lambda^2 x^2  \left|\ln\left(\frac{x}{2}\right)\right|.
\]
\end{prop}

\begin{proof}[Proof of Proposition~\ref{pro:n0}]
Let us first consider the case $\lambda<1$. Since $g_0$ is decreasing, for all $x<j_{0,1}$, we have
$$
g_0\left(x\right)-g_0\left(\lambda x\right)= (1-\lambda)x g_0^\prime(\zeta_0)<0,
$$
Then, for all $x\leq y_{0,1}$, we have $0< J_{0}(\lambda x) \leq 1$ and $0\leq u_{0}Y_{0}$. Consequently, 
\begin{equation}\label{eq:bdt01b}
\left|u_{0}(x)\left(J_{0}(x)+iY_{0}(x)\right)+i\frac{2}{\pi}J_{0}\left(\lambda x\right)\right|
\geq
\frac{2}{\pi}\left|J_{0}\left(\lambda x\right)\right|.
\end{equation}
Thanks to Proposition~\ref{pro:propg0k0}, for all $0<x<y_{0,1}$,
\begin{equation}\label{eq:bdsmall1}
\left| g_0\left(x\right)-g_0\left(\lambda x\right) \right|
\leq \min( (1-\lambda)x\left| g_0^\prime\left(x\right)\right|,\left| g_0^\prime\left(x\right)\right| 
\leq \min\left(\frac{4}{3} (1-\lambda),\frac{2}{3}\right) x^2.
\end{equation}
Note that for $0<x<1$,  there holds 
\begin{equation}\label{eq:bd-hoto}
\left|J_{0}(x)+iY_{0}(x)\right|<\frac{3}{2\sqrt{2}}\left|\ln\left(\frac{x}{2}\right)\right|.
\end{equation}
Together with \eqref{eq:bdsmall1}, this shows that
\begin{eqnarray}
|u_{0}(x)\left(J_{0}(x)+iY_{0}(x)\right)|&=& \left|J_{0}\left(\lambda x\right)\right| \left|J_{0}\left( x\right)\right|\left|J_{0}(x)+iY_{0}(x)\right| 
\left|g_0(x)-g_0(\lambda x)\right| \nonumber\\ 
&\leq&  \frac{1}{\sqrt{2}} \min( 1, 2-2\lambda) x^2 \left|\ln\left(\frac{x}{2}\right)\right| \left|J_{0}\left(\lambda x\right)\right|.\label{eq:bdt02}
\end{eqnarray}
Inserting the estimates \eqref{eq:bdt01b} and \eqref{eq:bdt02}  in formula \eqref{eq:formula-Sn-1}, we obtain
$$
 |S_0| \leq  \frac{\pi}{2\sqrt{2}} \min(1, 2-2\lambda)  x^2\left|\ln\left(\frac{x}{2}\right)\right| . 
$$ 
Let us now suppose $1\leq\lambda$. 
For all $x\leq \frac{1}{2\lambda}$, using the bounds on $g_0$ given by Proposition~\ref{pro:propg0k0}, we have
$$
\left| g_0\left(x\right)-g_0\left(\lambda x\right)\right| 
\leq \left| g_0\left(\lambda x\right) \right|\leq \frac{\lambda^2 x^2}{2}\left(1 + \frac{\lambda^2}{6} x^2 \right)\leq \frac{25}{48} \lambda^2 \,x^2
$$
Alternatively, note that for $x\leq m_\lambda $ and $\lambda\leq e^3$, that is, when $2> \sqrt{\ln(\lambda)+1}$, we can also bound
\begin{eqnarray*}
\left| g_0\left(x\right)-g_0\left(\lambda x\right)\right| 
&\leq& \frac{\lambda^2 x^2}{2}\left(1 -\frac{1}{\lambda^2}+ \left(\frac{\lambda^2}{6} -\frac{1}{8\lambda^2}\right) x^2 \right) \\
&\leq& \frac{\lambda^2 x^2}{2}\left(1 -\frac{1}{\lambda^2}+ \left(\frac{\lambda^2}{6} -\frac{1}{8\lambda^2}\right) \frac{1}{\lambda^2 \left(\ln (\lambda)+1\right)} \right)  \\
& < & \frac{25}{48} \lambda^2 \,x^2.
\end{eqnarray*}
Using that when $x<\frac{1}{2}$,
$$
\left|J_{0}(x)+iY_{0}(x)\right|<\frac{3}{4}\left|\ln\left(\frac{x}{2}\right)\right|
$$
and arguing as above, we obtain
\begin{equation}\label{eq:bdt03}
|u_{0}(x)\left(J_{0}(x)+iY_{0}(x)\right)|\leq  \frac{2}{5} x^2 \lambda^2 \left|\ln\left(\frac{x}{2}\right)\right| \left|J_{0}\left(\lambda x\right)\right|.
\end{equation}
Alternatively, starting from the inequality
$$
\left| g_0\left(x\right)-g_0\left(\lambda x\right)\right| 
\leq (\lambda-1)x\left| g_0^\prime\left(\lambda x\right) \right|\leq \frac{4}{3} \frac{\lambda-1}{\lambda} \lambda^2 x^2,
$$
we obtain
\begin{equation}\label{eq:bdt03b}
|u_{0}(x)\left(J_{0}(x)+iY_{0}(x)\right)|\leq  \frac{\lambda-1}{\lambda} x^2 \lambda^2 \left|\ln\left(\frac{x}{2}\right)\right| \left|J_{0}\left(\lambda x\right)\right|.
\end{equation}
Next, using the inverse triangular inequality,
\begin{equation}\label{eq:bdt04}
\left|u_{0}(x)\left(J_{0}(x)+iY_{0}(x)\right)+i\frac{2}{\pi}J_{0}\left(\lambda x\right)\right| 
\geq \left| \frac{2}{5} x^2 \lambda^2 \left|\ln\left(\frac{x}{2}\right)\right| -\frac{2}{\pi} \right|\left|J_{0}\left(\lambda x\right)\right|
\geq \frac{2}{5\pi}|\left|J_{0}\left(\lambda x\right)\right|,
\end{equation}
provided $x\sqrt{\left|\ln\left(\frac{x}{2}\right)\right|}<\sqrt{\frac{4}{\pi}}\frac{1}{\lambda}$. 
Inserting the estimates \eqref{eq:bdt03b} and \eqref{eq:bdt04}  in formula \eqref{eq:formula-Sn-1}, we obtain
$$
 |S_0| \leq  \pi\min\left(1,\frac{5}{2} \frac{\lambda-1}{\lambda}\right)    \lambda^2 x^2  \left|\ln\left(\frac{x}{2}\right)\right|
$$
Then remark that when $x<m_\lambda$, then $x\sqrt{\left|\ln\left(\frac{x}{2}\right)\right|}<\sqrt{\frac{4}{\pi}}\frac{1}{\lambda}$,
 for all $\lambda\geq1$. 
\end{proof}

\begin{prop}\label{pro:R0}
Let $R\geq\eps$. For any $\lambda>0$ and any $x>0$,
\begin{equation}\label{eq:bd-ho-s}
 \left| R_{0}^{\eps}\BH{0}{x \frac{R}{\eps}} \right|^{2} \leq  \frac{2}{\pi x} \frac{\eps}{R}.
\end{equation}
If $\lambda<1$ then for all $x>0$
$$
\left| R_{0}^{\eps}\BH{0}{x \frac{R}{\eps}} \right|  \leq  \left|\BH{0}{y_{0,1}\frac{R}{\eps}}\right|.
$$
Furthermore, when $x<y_{0,1}$,
$$
\left| R_{0}^{\eps}\BH{0}{x \frac{R}{\eps}} \right| \leq   \frac{\pi^2}{2\sqrt{2}} (1-\lambda)  x^2 \left|\BH{0}{x \frac{R}{\eps}}\right|. 
$$
If $1\leq\lambda$ then for all $x>0$ 
$$
\left| R_{0}^{\eps}\BH{0}{x \frac{R}{\eps}} \right| \leq 
\sqrt{5} \left|\BH{0}{\min\left(\frac{1}{2},m_{\lambda}\right) \frac{R}{\eps}} \right|.
$$
Furthermore, when  $x<\min\left(\frac{1}{2},m_{\lambda}\right)$,
$$
\left| R_{0}^{\eps}\BH{0}{x \frac{R}{\eps}} \right| \leq   \frac{5\pi^2}{4} \frac{\lambda-1}{\lambda}x^2 \lambda^2 \left|\BH{0}{x \frac{R}{\eps}}\right|.
$$
\end{prop}
\begin{proof}
Note that $x\left|\BH{0}{x}\right|^2$ is an increasing function of $x$, with limit $\frac{2}{\pi}$. Since  $\left|R_{0}^{\eps}\right|\leq 1$ for all $x>0$, this implies
\eqref{eq:bd-ho-s}. 

Suppose now $\lambda<1$. Note that $\BH{0}{\cdot}$ is decreasing, therefore using the simple bound  
$\left|R_{0}^{\eps}\right|\leq 1$ we have for all $x\geq y_{0,1}$,
\[
\left|R_{0}^{\eps} \BH{0}{x \frac{R}{\eps}}\right|^{2} \leq \left|\BH{0}{y_{0,1} \frac{R}{\eps}}\right|^2
\]  
For $0<x< 1$ it is easy to verify that,
\[ 
1+\left|\frac{\BY{0}{x}}{\BJ{0}{x}}\right|^{2}>1+ \frac{4}{\pi^{2}} \left|\ln\left(\frac{x}{2}\right)\right|^{2}.
\]
Therefore we obtain
$$
\left|R_{0}^{\eps}\BH{0}{x \frac{R}{\eps}}\right|^{2} 
\leq \frac{\left|{S_0}(x)\right|^2}{1+\frac{4}{\pi^{2}} \left|\ln\left(\frac{x}{2}\right)\right|^{2}}  \left|\BH{0}{x \frac{R}{\eps}}\right|^2 
$$
Thanks to Proposition~\ref{pro:n0}, we have, for all $0<x<y_{0,1}$,
\begin{equation}\label{eq:mono-growth}
\left|R_{0}^{\eps}\BH{0}{x \frac{R}{\eps}}\right|^{2} 
\leq
\frac{\pi^2}{8}  \min\left(1,2-2\lambda\right)^2 \frac{\eps}{R} (1-\lambda)^2 
\frac{x^3 \left|\ln\left(\frac{x}{2}\right)\right|^{2} }{1 + \frac{4}{\pi^2} \left|\ln\left(\frac{x}{2}\right)\right|^{2}} 
\left(x\frac{R}{\eps} \left|\BH{0}{x \frac{R}{\eps}}\right|^2\right).
\end{equation}
The function
$$
x\to \frac{x^3 \left|\ln\left(\frac{x}{2}\right)\right|^{2} }{1 + \frac{4}{\pi^2} \left|\ln\left(\frac{x}{2}\right)\right|^{2}}
$$
is increasing on $(0,1)$, and as we noted before, so is $x\left|\BH{0}{x}\right|^2$. 
Therefore an upper bound is obtained by choosing $x=y_{0,1}$ in the right-hand-side of 
\eqref{eq:mono-growth}, which gives
$$
\left|R_{0}^{\eps}\BH{0}{x \frac{R}{\eps}}\right|^{2} 
\leq \frac{\pi^2}{8}  \frac{ y_{0,1}^4 \ln\left(\frac{y_{0,1}}{2}\right)^2}{1 + \frac{4}{\pi^2} \ln\left(\frac{y_{0,1}}{2}\right)^2} 
\left|\BH{0}{y_{0,1} \frac{R}{\eps}}\right|^2 \leq \left|\BH{0}{y_{0,1} \frac{R}{\eps}}\right|^2,
$$
where we used \eqref{eq:bd-ho-s} in the second inequality. We have obtained that for all $x>0$,
$$
\left|\Ro \BH{0}{x \frac{R}{\eps}}\right|\leq \left|\BH{0}{y_{0,1} \frac{R}{\eps}}\right|.
$$
Alternatively, note that the function $(\ln(x/2)^2)/(1 + \frac{4}{\pi^2} \ln(x/2)^2)$ is decreasing on $(0,1)$, 
with a maximum of $\pi^2/4$, therefore \eqref{eq:mono-growth} and \eqref{eq:bd-ho-s} yield
$$
\left| R_{0}^{\eps}\BH{0}{x \frac{R}{\eps}} \right|^{2} \leq \frac{\pi^4}{8} (1-\lambda)^2  x^4 \left|\BH{0}{x \frac{R}{\eps}}\right|^2,
$$
for all $0<x<y_{0,1}$. 

\medskip{}

Let us now consider the case $\lambda>1$. We only consider the case when $m_{\lambda}\geq2$, 
the proof in the other case is similar. Arguing as before, we have for all $x$ such that $x\geq m_{\lambda}$
\[
\left|R_{0}^{\eps} \BH{0}{x \frac{R}{\eps}}\right|\leq \left|\BH{0}{m_{\lambda}\frac{R}{\eps}} \right|.
\]
Next, when $x<m_{\lambda}$, we have
\begin{align*}
\left|R_{0}^{\eps} \BH{0}{x \frac{R}{\eps}}\right|^{2}
&\leq \frac{\left|{S_0}(x)\right|^2}{1+\frac{4}{\pi^2}\left|\ln\left(\frac{x}{2}\right)\right|^{2}} \left|\BH{0}{x \frac{R}{\eps}}\right|^2 \\
&\leq \min\left(1, \frac{5}{2} \frac{\lambda-1}{\lambda}\right)^2\pi^2 \frac{x^4 \lambda^4 \left|\ln\left(\frac{x}{2}\right)\right|^2}{ 1+\frac{4}{\pi^2}\left|\ln\left(\frac{x}{2}\right)\right|^2} 
\left|\BH{0}{x \frac{R}{\eps}}\right|^2
\end{align*}
Arguing as in the case $\lambda<1$, an upper bound is obtained by replacing $x$ by its maximal value, namely
\begin{eqnarray}
 \left|R_{0}^{\eps} \BH{0}{x \frac{R}{\eps}}\right|^{2} &\leq& \pi^2 
 \frac
 {\left(\ln\left(2\lambda\sqrt{\ln\lambda+1}\right)\right)^2}
 {\left(1+\ln(\lambda)\right)^2\left(1+\frac{4}{\pi^2}\left(\ln\left(2\lambda\sqrt{\ln\lambda+1}\right)\right)^2 \right)}
 \left|\BH{0}{\frac{1}{\lambda \sqrt{\ln \lambda +1}} \frac{R}{\eps}}\right|^2 \nonumber\\
&\leq& \frac{5}{1+\left(\frac{2}{\pi^2} \ln \lambda\right)^2} \left|\BH{0}{m_{\lambda} \frac{R}{\eps}}\right|^2.\label{eq:opti-bd-lb-os}
\end{eqnarray}
We have obtained that for all $x>0$,
\[
\left|R_{0}^{\eps} \BH{0}{x \frac{R}{\eps}}\right|^{2}
\leq 5 \left|\BH{0}{m_{\lambda}\frac{R}{\eps}} \right|^2.
\]
Alternatively, we also have, when $x<m_{\lambda}$, 
\begin{align*}
\left|R_{0}^{\eps} \BH{0}{x \frac{R}{\eps}}\right|^{2}
&\leq \pi^2  \left(\frac{5}{2} \frac{\lambda-1}{\lambda}\right)^2 
     \frac{x^4 \lambda^4 \left|\ln\left(\frac{x}{2}\right)\right|^2}{ 1+\frac{4}{\pi^2}\left|\ln\left(\frac{x}{2}\right)\right|^2} 
     \left|\BH{0}{x \frac{R}{\eps}}\right|^2 \\
&\leq  \left(\frac{5\pi^2}{4} \frac{\lambda-1}{\lambda}x^2 \lambda^2 \right)^2  \left|\BH{0}{x\frac{R}{\eps}} \right|^2.
\end{align*}
\end{proof}
We conclude this section by an estimate which will prove useful for broadband estimations.
\begin{prop}\label{pro:R0plus}
Let $R\geq\eps$. For any $\lambda\geq 7$ and any $m_\lambda\geq x>0$,
\begin{equation}\label{eq:bd-rhop}
\left| R_{0}^{\eps}\BH{0}{x \frac{R}{\eps}} \right| \leq   4  \left|\frac{\BH{0}{m_\lambda \frac{R}{\eps}}}{\BH{0}{m_\lambda}}\right|.
\end{equation}
\end{prop}
\begin{proof}
Inequality \eqref{eq:opti-bd-lb-os} shows that for all $0<x\leq m_\lambda$, we have
$$
\left|R_{0}^{\eps} \BH{0}{x \frac{R}{\eps}}\right| \leq \sqrt{\frac{5}{1+\left(\frac{2}{\pi^2} \ln \lambda\right)^2}} 
\left|\BH{0}{m_{\lambda} \frac{R}{\eps}}\right|.
$$
We note (and can prove, it is a study of a function of one variable, $\ln \lambda$) that
$$
\lambda \to \frac{\left|\BH{0}{m_{\lambda}}\right|}{\sqrt{1+\left(\frac{2}{\pi^2} \ln \lambda\right)^2}} 
$$
is increasing for $\lambda>7$, and its value at $\lambda=7$ is greater than $3.2$. This lower bound yields our estimate. 
\end{proof}

\section{\label{sec:case2} Bounds near quasi-resonances}

This section is devoted to the proof of Proposition~\ref{pro:Ito-One}.

We wish to bound the size of the "blow-up" regions, that is, the sets $I_{n,k}(\tau)$ defined by \eqref{eq:def-itau}, centered on quasi-resonances. From \eqref{eq:Dixon-JnYn} we know that  $y_{n,1}^{(1)}>y_{n,1}$ for all $n\geq0$, 
thus $k_n>0$ on $(0,y_{n,1})$. Introducing
\begin{equation}\label{eq:def-phin}
\phi_n:= \begin{array}[t]{rcl}
   (0,y_{n,1})\setminus \cup_k \{j_{n,k}/\lambda\} &\to& \mathbb{R} \\
   x     &\to& \displaystyle \frac{g_n \left(\lambda x\right)}{k_n (x)}, 
\end{array}
\end{equation}
we have
$$
\phi_n(I_{n,k}(\tau)) = [-1-\tau,-1+\tau].
$$
We first verify that $\phi_n$ is one-to-one on $I_{n,k}(\tau)$, for $\tau$ small enough and $\lambda$ large enough.
\begin{lem}\label{lem:phin}
Suppose $\tau\leq\frac{1}{4}$ and $7\leq\lambda$. When $n\geq 1$ 
the function $\phi_n$ given by \eqref{eq:def-phin} satisfies
$$
\phi_n^\prime(x) \leq \frac{1-\lambda^2}{2 n_{+} k_n(x)} <0
$$
where $n_{+}=\max(n,1)$, for all $x\in I_{n,k}(\tau)\cap(0,n)$ when $n\geq1$ and 
for all $x\in I_{0,k}(\tau)\cap(0,\zeta_0)$ when $n=0$. Furthermore, 
$$
I_{0,k}(\tau) \subset \left(m_{\lambda}, y_{0,1}\right), \mbox{ and for } n\geq1,\quad I_{n,k}(\tau) \subset 
\left(\frac{j_{n,1}^{(1)}}{\lambda}, y_{n,1}\right),
$$
for all $k$.
\end{lem}
\begin{proof}
We compute, using \eqref{eq:difeqgn} and \eqref{eq:difeqkn}, for $n\geq0$,
\begin{equation} \label{eq:phip-2}
 \phi_n^\prime (x) 
= \frac{x}{n_{+} k_n(x)} \left (1 - \lambda^2\right) 
+ \frac{k_n(x) +g_n(\lambda x)}{k_n(x)^2} 
\left( \frac{n}{x} - \frac{x}{n_{+}} - \frac{n_{+}}{x} g_n (\lambda x) k_n(x)\right),
\end{equation}
Suppose first that $n\geq1$. When  $x\leq n$ and $- g_n > k_n$, 
$$
(k_n(x) +g_n(\lambda x))(\frac{n}{x} - \frac{x}{n} - \frac{n}{x} g_n (\lambda x) k_n(x)) <0 
$$
therefore
$$
\phi_n^\prime (x) 
\leq \frac{x}{n k_n(x)} \left(1-\lambda^2\right) <0.
$$
On the other hand, when $g_n(\lambda x) + k_n(x) >0$ and $x\in I_{n,k}(\tau)$, that is, when 
$0 < g_n(\lambda x) + k_n(x) \leq \tau k_n(x)  $, we have
$$
\frac{k_n(x) +g_n(\lambda x)}{k_n(x)}
\left( \frac{n}{x} - \frac{x}{n} - \frac{n}{x} g_n (\lambda x) k_n(x)\right) \leq \frac{\tau}{n\, x} \left( n^2  - x^2 + n^2 k_n^2 \right)  
$$
Using the upper bound on $k_n$ given by Proposition~\ref{pro:propsg}, we find that when $\frac{n}{\lambda}\leq x\leq \kappa_n$, we have 
$$
\frac{2\, \tau}{n\, x} \left( n^2  - x^2  + n^2 k_n^2  \right)  \leq  2\, \tau \frac{ x}{n} \frac{\lambda^2 -1}{\lambda} 
\leq \frac{1}{2} \frac{ x}{n} \left(\lambda^2-1\right)
$$
provided $\tau \leq \frac{1}{4}$. When $\kappa_n\leq x \leq n$, 
$$
\frac{2\, \tau}{n\, x} \left( n^2  - x^2  + n^2 k_n^2  \right)  \leq  2 \tau \frac{ x}{n} \frac{ n^{2} +7/6 n^{4/3}}{ (n -4/5 n^{1/3})^2}   
\leq \frac{93 \tau}{\lambda^2-1} \frac{ x}{n} \frac{\lambda^2 -1}{\lambda}  <  \frac{1}{2} \frac{ x}{n} \left(\lambda^2-1\right),
$$
when $\tau\leq\frac{1}{4}$ and $7\leq \lambda$. We have obtained that when $x\in I_{n,k}(\tau)$,
$$
\phi_n^\prime (x) 
\leq \frac{x}{2 n k_n(x)} \left(1-\lambda^2\right) <0,
$$
as announced. Finally, note that at $x=j_{n,1}^{(1)}/\lambda$ we have  $g_n(\lambda x)=0 > (\tau-1) k_n(x)$, thus $I_{n,k}(\tau)$ is a proper subset of $U_{n,k}$.  
 Let us now consider the case $n=0$. We have
\begin{equation} \label{eq:phip-0}
 \phi_0^\prime (x) 
= \frac{x}{k_0(x)} \left (1 - \lambda^2\right) 
+ x \frac{k_0(x) +g_0(\lambda x)}{  k_0(x)^2} 
\left(  - 1  - \frac{1}{x^2} g_0 (\lambda x) k_0(x)\right).
\end{equation}
When $-\tau k_0(x) \leq g_0 (\lambda x) + k_0(x) \leq \tau k_0(x)$, thanks to Proposition~\ref{pro:logYn}, we have
$$
- \frac{1}{x^2} g_0 (\lambda x) k_0(x) -1  \geq (1-\tau) \frac{1}{(0.36)^2} -1 > 0,
$$
when $\tau \leq \frac{1}{4}$. Turning back to \eqref{eq:phip-0}, this shows that when $k_0(x) + g_0(\lambda x)\leq0$,
$$
\phi_0^\prime (x) \leq \frac{x}{k_0(x)} \left (1 - \lambda^2\right) .
$$ 
Let us now assume $0< k_0(x) + g_0(\lambda x)  < \tau k_0(x)$. 
We claim that on $(0,m_{\lambda})$, $-\phi_0 < 3/5$. Therefore no $I_{0,k}(\tau)$ lies in the interval $(0,m_{\lambda})$, since $\tau\leq 1/4$. 
Using the bounds on $g_0$ and $k_0$ given by Proposition~\ref{pro:propg0k0} we find, when $ x \leq m_{\lambda}$ and $\lambda \geq 7$, 
maximizing in $x$ first and then in $\lambda$,
\begin{eqnarray*}
\frac{-g_0(\lambda x)}{k_0(x)} &\leq& \frac{\lambda^2 x^2}{2}\ln \left(\frac{2}{x e^\gamma}\right) \left(1 + \frac{1}{12} \lambda^4 x^4\right) \\
&\leq& \frac{1}{2 \ln \lambda +2}\ln \left(\frac{2\lambda\sqrt{\ln \lambda + 1}}{e^\gamma} \right) \left((1 + \frac{1}{12} \frac{1}{(\ln \lambda +1)^2}\right) \\
&\leq& \frac{1}{2} \left(1 + e^{-3 - 2\gamma}\right)\left((1 + \frac{1}{12} \frac{1}{(\ln 7 +1)^2}\right) <\frac{3}{5}, \\
\end{eqnarray*}
which is our claim. This is turn shows that when $x\in I_{0,k}(\tau)$, and $x\leq \zeta_0$, and $7\leq \lambda$,
\begin{equation}\label{eq:bd-up-ml}
-\frac{1}{x^2}g_0(\lambda x) k_0(x) 
 \leq \frac{1}{m_{\lambda}^2} k_0^2(m_{\lambda}) 
 \leq \left(\frac{1}{m_{\lambda} \left(\ln \left(\frac{2}{e^\gamma}\right) -\ln( m_{\lambda})\right)} +m_{\lambda}\right)^2 < \frac{\lambda^2}{\ln(\lambda)}.
\end{equation}
We therefore have obtained that for all $x\in I_{0,k}(\tau)\cap(0,\zeta_0)$,
$$
 x \frac{k_0(x) +g_0(\lambda x)}{  k_0(x)^2} 
 \left(  - 1  - \frac{1}{x^2} g_0 (\lambda x) k_0(x)\right) \leq  \tau \frac{x}{k_0(x)}(\lambda^2-1), 
$$
which in turn shows that
$$
\phi_0^\prime (x) \leq \frac{x}{2 k_0(x)} \left (1 - \lambda^2\right).
$$ 
\end{proof}

We are now ready to compute an upper bound on the sum of size of the intervals $I_{n,k}(\tau)$ for a given $n$.
\begin{prop}
\label{pro:Ink} Suppose $\lambda\geq 7$ and $\frac{1}{4}\geq \tau$.
For all $n \geq 1$, 
\begin{equation}\label{eq:boundInk}
\left|\bigcup_{k\in K(\lambda,n)} I_{n,k}(\tau)\right| \leq  6 \tau \frac{n \ln \lambda}{\lambda},
\end{equation}
where $K(\lambda,n)$ is the set of all positive $k$ such that $j_{n,k}^{(1)}< n \lambda$.
We also have
\begin{equation}\label{eq:boundIn0}
\left|\bigcup_{k\in K(\lambda,0)} I_{0,k}(\tau)\right| \leq  7\tau \frac{\ln(\ln \lambda)}{\lambda},
\end{equation}
where $K(\lambda,0)$ is the set of all positive $k$ such that $j_{0,k}^{(1)}< \zeta_0 \lambda$.
\end{prop}
\begin{proof}
{\bf When $n\geq 1$.} We know thanks to Lemma~\ref{lem:phin} that  $\phi_{n}$ is a bijection on $I_{n,k}(\tau)$. 
We write  $I_{n,k}(\tau)=[\alpha_{n,k},\beta_{n,k}]$, that is $\beta_{n,k}\in U_{n,k}$ is such that $\phi(\beta_{n,k}) = -1-\tau$
 and $\alpha_{n,k}\in U_{n,k}$ is such that $\phi(\beta_{n,k}) = -1+\tau$. 
We have
\begin{equation}\label{eq:bd-tau-1n}
2\, \tau = \phi(\alpha_{n,k}) - \phi(\beta_{n,k}) = \int_{\alpha_{n,k}}^{\beta_{n,k}} - \phi^\prime (u) du 
\geq \frac{\lambda^2 -1}{2} \int_{\alpha_{n,k}}^{\beta_{n,k}}  \frac{u}{n k_n(u)} du.
\end{equation}
From Proposition~\ref{pro:logYn} we know that, introducing $\chi_n=n-4/5n^{1/3}$, 
$u\to u/k_n(u)$ is increasing on $(0,\chi_n)$, and satisfies 
$$
\frac{x}{n k_n(x)} \geq \begin{cases}
                         \left(\frac{n^2}{x^2}-1\right)^{-1/2} & \mbox{ when } 0\leq x\leq \chi_n, \\
                          \left(\frac{n^2}{\chi_n^2}-1\right)^{-1/2} & \mbox{ when } \chi_n\leq x \leq n.
                        \end{cases}
$$
Let $k_{*}$ be the largest indice such that $\alpha_{n,k}<\chi_n$. For all $k\leq k_*$ we have 
$$
\frac{4\tau}{\lambda^2-1} \geq  \int_{\alpha_{n,k}}^{\beta_{n,k}} \left(\frac{n^2}{x^2}-1\right)^{-1/2} dx \geq \left|I_{n,k}(\tau)\right|\left(\frac{n^2}{\alpha_{n,k}^2}-1\right)^{-1/2}
$$
and when $\alpha_{n,k} > \chi_n$, 
$$
\frac{4\tau}{\lambda^2-1} \geq  \int_{\alpha_{n,k}}^{\beta_{n,k}} \left(\frac{n^2}{x^2}-1\right)^{-1/2} dx \geq \left|I_{n,k}(\tau)\right|\left(\frac{n^2}{\chi_{n,k}^2}-1\right)^{-1/2}
$$
Observing that $I_{n,k}(\tau)\subset U_{n,k}$, and therefore $n>j_{n,k}/\lambda> \alpha_{n,k} > j_{n,k}^{(1)}/\lambda > \alpha_{n,0}:=n/\lambda$ we have obtained that
 \begin{eqnarray*}
 \left|\bigcup_{k=1}^{k_*} I_{n,k}(\tau)\right| &\leq& \frac{4 \tau}{\lambda^2-1} \sum_{k=1}^{k_*}\sqrt{\left(\frac{n}{\alpha_{n,k}}\right)^2 -1} \\
 &\leq& \frac{4 \tau  }{\lambda^2-1} \left(\min_{k\leq k*} \left(\frac{\alpha_{n,k}}{n} -\frac{\alpha_{n,k-1}}{n}\right)\right)^{-1}
 \sum_{k=1}^{k_*}\sqrt{\left(\frac{n}{\alpha_{n,k}}\right)^2 -1}\left(\frac{\alpha_{n,k}}{n} -\frac{\alpha_{n,k-1}}{n}\right)  \\
 &\leq& \frac{4  n \tau  }{\lambda^2-1} \left(\min_{k\leq k*}\left(\alpha_{n,k} - \alpha_{n,k-1}\right)\right)^{-1}
 \int_{\lambda^{-1}}^{1}\sqrt{\frac{1}{x^2} -1} \, dx.\\
&=& \frac{4 \tau n \lambda \ln \lambda }{\lambda^2-1} \, \max_{k\leq k*}\frac{1}{\left(\lambda\alpha_{n,k}-\lambda\alpha_{n,k-1}\right)}.
 \end{eqnarray*}
For $k\geq1$, the distance  $\lambda\alpha_{n,k+1} -\lambda\alpha_{n,k}$ is at least $j_{n,k+1}^{(1)} - j_{n,k}$. 
We know from \eqref{eq:prothetan} that this distance decreases with $k$, and tends to $\pi/2$. 
On the other-hand, using the estimates \eqref{eq:abn-wats}, 
$\lambda\alpha_{n,1} -\lambda\alpha_{n,0} > j_{n,1}^{(1)} -n > \frac{4}{5} n^{1/3} >\frac{4}{5}$. Therefore, we have 
\begin{equation}\label{eq:estim-ink-1}
\left|\bigcup_{k=1}^{k_*} I_{n,k}(\tau)\right| \leq  \frac{5 \tau n \lambda \ln \lambda }{\lambda^2-1}.
\end{equation}
Using again the fact that $j_{n,k}^{(1)} - j_{n,k-1}^{(1)}$ is at least $\pi$, there can be at most  $(n-\chi_n)\lambda/\pi$ intervals $U_{n,k}$ in $(\chi_n,n)$.
Therefore 
\begin{eqnarray*}
\left|\bigcup_{k\geq k_*} I_{n,k}(\tau)\right| &\leq& \frac{4 \lambda \tau}{\pi(\lambda^2 -1) } \sqrt{\frac{n^2}{\chi_n^2} -1 }(n-\chi_n)\\
 &\leq& \frac{4 n\lambda \tau}{\pi(\lambda^2 -1) }.
\end{eqnarray*}
Altogether, we have
$$
\left|\bigcup_{k\in K(\lambda,n)} I_{n,k}(\tau)\right| \leq \frac{n \tau \ln \lambda}{\lambda} 
\left(\frac{\lambda^2}{\lambda^2-1}\left( 5 + \frac{4}{\pi \ln \lambda}\right) \right) \leq  \frac{ 6 n \tau \ln \lambda}{\lambda} 
$$
which completes the proof of estimate \eqref{eq:boundInk}. 

\medskip{}

{\bf When $n=0$}. As above, with the same notations, we have
\begin{equation}\label{eq:bd-tau-0n}
2\, \tau = \phi(\alpha_{0,k}) - \phi(\beta_{0,k}) = \int_{\alpha_{0,k}}^{\beta_{0,k}} - \phi^\prime (u) du 
\geq \frac{\lambda^2 -1}{2} \int_{\alpha_{0,k}}^{\beta_{0,k}}  \frac{u}{k_0(u)} du.
\end{equation}
Proposition~\ref{pro:logYn} shows that $x\to x/k_0(x)$ is increasing until $\zeta_0$, and decreasing afterwards. Thus
$$
2\, \tau \geq \frac{\lambda^2 -1}{2} \left| I_{0,k}(\tau) \right| \frac{\alpha_{0,k}}{k_0(\alpha_{0,k})}.
$$
As before, we know that for $k\geq 1$, $\alpha_{0,k} - \alpha_{0,k+1} > \lambda^{-1} (j_{0,k} - j_{0,k+1}^{(1)}) > \lambda^{-1} \pi/2$. We have
\begin{eqnarray*}
\sum_{k\geq2} \frac{k_0(\alpha_{0,k})}{\alpha_{0,k}}  &\leq& \max_{k\geq1}\frac{1}{\alpha_{0,k} - \alpha_{0,k+1}} 
\sum_{k\geq2} \frac{k_0(\alpha_{0,k})}{\alpha_{0,k}} \left(\alpha_{0,k-1} - \alpha_{0,k}\right) \\
&\leq& \frac{2\lambda}{\pi} \int_{j_{0,1}/\lambda}^{\zeta_0} \frac{k_0(x)}{x} dx \\
&\leq& \frac{2\lambda}{\pi} \ln\left|\frac{\BY{0}{j_{0,1}/\lambda}}{\BY{0}{\zeta_{0}}}\right|
\leq \frac{2}{\pi} \lambda \ln\left(\ln \lambda\right).
\end{eqnarray*}
Finally, using the bound \eqref{eq:bd-up-ml}
$$
\frac{\alpha_{0,1}}{k_0(\alpha_{0,1})} \leq \frac{\lambda}{\sqrt{\ln \lambda}}
$$
and we have obtained that 
$$
\left|\bigcup_{k\in K(\lambda,0)} I_{0,k}(\tau) \right| \leq \frac{4 \lambda \tau}{\lambda^2 -1} 
\left( \frac{1}{\sqrt{\ln \lambda}} + \frac{2}{\pi}  \ln\left(\ln \lambda\right)\right) 
\leq \tau \frac{7 \ln \left(\ln \lambda\right)}{\lambda},
$$
which concludes our proof.
\end{proof}

Let us now check that  away from $\omega_{n,k}$, we
can produce a bound ${S_n}$ similar to that of the perturbative regime.
\begin{prop}
\label{pro:Ito}
If $n\geq 1$ and  $x \in \left(\lambda^{-1}j_{n,1}^{(1)},y_{n,1}\right)\setminus(\cup_{k} I_{n,k}(\tau))$, there holds
\[
\left|{S_n}\right|\leq\frac{9}{2\,\tau}.
\]
When $n=0$, if $x\in \left(m_\lambda,\zeta_0\right)\setminus(\cup_{k} I_{0,k}(\tau))$, we have
\[
\left|{S_0}\right|\leq \frac{5}{3\,\tau}.
\]
\end{prop}
\begin{proof}
{\bf Case $n\neq0$.}   When $x\in\left[n,y_{n,1}\right]$ Proposition~\ref{pro:ntoyn1} shows that $\left|{S_n}\right|\leq\sqrt{5}$, which establishes the bound.  
The proof is along the lines of that of Lemma~\ref{lem:FirstCase}. Starting from the formula
\[
{S_n}(x)=\frac{\left(g_{n}(\lambda x)-g_{n}(x)\right)\left(1+i\tan\theta_{n}\right)}{\left(g_{n}(\lambda x)-g_{n}(x)\right)+i\tan\theta_{n}\left(g_{n}(\lambda x)+k_{n}(x)\right)},
\]
we write $a=g_{n}(x)$ and $b=k_{n}(x)$, and the study of the function
\[
u\to\frac{(u-a)^{2}\left(1+\tan^{2}\theta_{n}\right)}{(u-a)^{2}+\tan^{2}\theta_{n}(u+b)^{2}}
\]
for $a>0$ and $b>0$ , with $u\in(-\infty,-(1+\tau)b)\cup(-(1-\tau)b,+\infty)$, 
shows that it has a minimum for $u=a$, tends to $1$ for $u\to\pm\infty$,
increases until $-(1+\tau)b$, decreases on $(-(1-\tau)b,a)$ and
increases to $1$ afterwards. Therefore, the maximum of ${S_n}$ is smaller
than the maximum of the two values $A$ and $B$ given by
\begin{align*}
A & =\frac{\left(1+\tan^{2}\theta_{n}\right)\left(\left(1+\tau\right)b+a\right)^{2}}{\left(\left(1+\tau\right)b+a\right)^{2}+\tan^{2}\theta_{n}\tau^{2}b^{2}}=\frac{1+\tan^{2}\theta_{n}}{1+\tau^{2}\tan^{2}\theta_{n}\frac{b^{2}}{a^{2}}\left(1+(1+\tau)\frac{b}{a}\right)^{-2}},\\
B & =\frac{\left(1+\tan^{2}\theta_{n}\right)\left(\left(1-\tau\right)b+a\right)^{2}}{\left(\left(1-\tau\right)b+a\right)^{2}+\tan^{2}\theta_{n}\tau^{2}b^{2}}=\frac{1+\tan^{2}\theta_{n}}{{\displaystyle 1+\tau^{2}\tan^{2}\theta_{n}\frac{b^{2}}{a^{2}}\left(1+(1-\tau)\frac{b}{a}\right)^{-2}}}.\end{align*}
It is clear that\[
\frac{\tau}{1+(1+\tau)\frac{b}{a}}<\frac{\tau}{1+(1-\tau)\frac{b}{a}}<1\]
therefore the maximum is $A$. Noticing that, for $b>\alpha a$
\[
\frac{\frac{b}{a}\tau}{1+(1+\tau)\frac{b}{a}}>\frac{\alpha\tau}{1+2\alpha}\]
we obtain\[
A<\left(\frac{1+2\alpha}{\alpha\tau}\right)^{2}\frac{1+\tan^{2}\theta_{n}}{\frac{1+2\alpha}{\alpha\tau}+\tan^{2}\theta_{n}}<\left(\frac{1+2\alpha}{\alpha\tau}\right)^{2}\]
Thanks to Proposition~\ref{pro:propsg} we know that $\frac{2}{5}<\frac{b}{a}.$
We can therefore conclude that 
\[
A\leq\left(\frac{9}{2}\frac{1}{\tau}\right)^{2},
\]
which concludes the proof. 

\smallskip{}

{\bf Case $n=0$.} The proof is slightly different when $n=0$, as $k_0/g_0$ is unbounded near $x=0$. 
Note that when $x<\frac{1}{3}$, (and $\zeta_0<1/3$),
$$
 \frac{3}{5} \frac{1}{k_0(x)} \leq \tan \theta_0(x)  \leq \frac{4}{5} \frac{1}{k_0(x)} \quad \frac{9}{10} \leq \frac{k_0(x)+g_0(x)}{k_0(x)}\leq 1.
$$ 
Introducing $u=\frac{g_0(\lambda x)+ k_0(x)}{k_0(x)}$, and $v=(k_0(x)+g_0(x))/k_0(x)$ we have 
$$
u\in(-\infty,-\tau)\cup(\tau,\infty)\mbox{ and } \frac{9}{10} \leq v \leq 1,
$$
\begin{eqnarray*}
|S_0(x)|^2&=&\frac{\left|g_{0}(\lambda x)-g_{0}(x)\right|^2\left(1+\tan^2\theta_{0}(x)\right)}
{\left|g_{0}(\lambda x)-g_{0}(x)\right|^2 +\tan^2\theta_{0}(x)\left|g_{0}(\lambda x)+k_{n}(x)\right|^2} \\
&\leq& \frac{\left|u - v\right|^2\left(\frac{4}{5}+k_0(x)^2\right)}
{\left|u-v\right|^2 k_0(x)^2 + \frac{3}{5} u^2}.
\end{eqnarray*}
Relaxing $u$ to an independent variable, we see that
$$
\frac{\left|u - v\right|^2\left(\frac{4}{5}+k_0(x)^2\right)}
{\left|u-v\right|^2 k_0(x)^2 + \frac{3}{5} u^2} \leq \max(A,B,1)
$$
where
$$
A =\frac{\left|\tau - v\right|^2\left((\frac{4}{5})^2+k_0(x)^2\right)}
{\left|\tau-v\right|^2 k_0(x)^2 + (\frac{3}{5})^2 \tau^2} \mbox{ and }
B = \frac{\left|\tau + v\right|^2\left((\frac{4}{5})^2+k_0(x)^2\right)}
{\left|\tau + v\right|^2 k_0(x)^2 + (\frac{3}{5})^2 \tau^2} \\
$$
It is clear that $A<B$. Taking the maximum value for $v$ and $\tau$, we find 
$$
B \leq \frac{1 +(\frac{5}{4})^2k_0(x)^2}
{\frac{5}{4} k_0(x)^2 + (\frac{3}{5})^2 \tau^2} \leq \frac{25}{9}\frac{1 +\left(\frac{5}{4} k_0(m_\lambda)\right)^2}
{\left(\frac{25}{12} k_0(m_\lambda)\right)^2 + \tau^2} \leq \frac{25}{9 \,\tau^2}.
$$
\end{proof}

\section*{Acknowledgement}
The subject of this paper was suggested by Michael Vogelius, who spent a considerable amount of time discussing this paper with 
the author and made several luminous comments which transformed this paper considerably.
 The author is extremely grateful for this, and would like to thank him profoundly.
 The author is supported by the EPSRC Science and Innovation award to the Oxford Centre for Nonlinear PDE (EP/E035027/1). 
This paper was written in part during the author's stay in the Institute for Advanced Study, and he would like to 
gratefully acknowledge the fantastic opportunity that this year was for him. 
This material is based upon work supported by the National Science Foundation under agreement No. DMS-0635607. 
Any opinions, findings and conclusions or recommendations expressed in this material are those of the authors 
and do not necessarily reflect the views of the National Science Foundation.

\bibliographystyle{plain}
\bibliography{/home/capdeboscq/Desktop/Dropbox/TeX/Mybib}
\appendix
\section{\label{ap:A} Miscellaenous Properties of Bessel Functions}

Proposition~\ref{pro:propsg} details properties of quotients of Bessel functions. This result is therefore independent from the rest of the paper. 
Properties {ii.}, {v.} and {vi.} could be new. Since many authors have worked on Bessel Functions, it is quite possible that similar results were 
proved before, but we are not aware of it. 

\begin{prop}
\label{pro:propsg}For all real $n\geq 1$, 
\begin{enumerate}
\item[i.] The function $g_{n}$ is strictly decreasing on $[0,\infty)\setminus\cup_{k}{j_{n,k}}$, and cancels at each $j_{n,k}^{(1)}$. 
\item[ii.] On $[0,j_{n,1})$, $g_{n}$ is concave. 
\item[iii.] For $0\leq x \leq n$, $g_n(x)$ satisfies  
\[
\sqrt{1-\left(\frac{x}{n}\right)^{2}}<g_{n}(x)<\sqrt{1-\left(\frac{x}{n}\right)^{2}+\frac{c_n^2}{n^{2/3}}\frac{x}{n}},
\]
where $c_n:=n^{1/3}g_{n}(n)$, satisfies
\[
\frac{1}{\sqrt{2}}<c_{n}<\frac{13}{14}.
\]
\item[iv.]  For $0< x < n$, $0 < -g_n^\prime(x) \leq g_n^2(n)$.  On $(0,y_{n,1})$, the function $k_n$ is positive. 
It decreases until $\kappa_n \in ( n -\frac{4}{5} n^{1/3},n)$, 
defined as the unique solution of 
$$
k_n(\kappa_n) = \sqrt{ 1- \frac{1}{n^2}\kappa_n^2}.
$$
\item[v.] For $0< x < n$, $k_n$  satisfies,
\[
\frac{3}{5} \frac{1}{n^{1/3}} \leq k_{n}(x)\leq \max\left(\sqrt{1-\left(\frac{x}{n}\right)^{2}},\frac{7}{6}\frac{1}{n^{1/3}}\right).
\]
More precisely, for all $x\leq \kappa_n$,
$$
\kappa^{+}\sqrt{1-\left(\frac{x}{n}\right)^{2}} - g_n(x) \leq k_n(x) \leq \sqrt{1-\left(\frac{x}{n}\right)^{2}},
$$ 
with $\kappa^{+}>1.91$, whereas for all $ \kappa_n \leq x \leq n$,
$$
\frac{3}{5} \frac{1}{n^{1/3}} \leq k_n(x) \leq \frac{7}{6}\frac{1}{n^{1/3}}
$$
\item[vi.] Finally, 
\begin{equation}
\frac{2}{5}<\frac{k_{n}(x)}{g_{n}(x)}<\frac{5}{3}.\label{eq:boundknovergn}
\end{equation}
\end{enumerate}
\end{prop}

\begin{proof}
Property (i) is well-known, see e.g. \cite{THIRUNAN-51,LANDAU-99}. To obtain property (ii), notice that 
\begin{eqnarray*}
-n g_n^\prime(x)&=&\frac{\left(\BJ{n}{x}\right)^2-\BJ{n+1}{x}\BJ{n-1}{x}}{\left(\BJ{n}{x}\right)^2} \\
             &=& \frac{1}{n+1} + \frac{2}{n+2} \left(\frac{\BJ{n+1}{x}}{\BJ{n}{x}}\right)^2\\ 
             & &                  + 2 n \sum_{p=2}^\infty \left(\frac{\BJ{n+p}{x}}{\BJ{n}{x}}\right)^2 \frac{1}{n+p-1}\frac{1}{n+p+1}.
\end{eqnarray*}

The second identity is proved in  \cite{SZASZ-50,THIRUNAN-51}.  Since for $x>0$, $x\to\BJ{n+1}{x}/\BJ{n}{x}$ is an increasing function
 as it can be readily observed from its continued fraction expansion, see \cite[10.10]{NIST-10}, we deduce that $g_n$ is concave on $(0,j_{n,1})$.
For property (iii), using the recurrence relations for Bessel functions, we notice that $g_n$ satisfies the differential equation  
\eqref{eq:difeqgn}.
Since $g_n$ is decreasing, we deduce that 
$$
g_n(x)\geq \sqrt{1-  \frac{x^2}{n^2}}.
$$
for $x\leq n$. We know \cite[8.55]{WATSON-25} that 
$ n\to n^{1/3} g_n(n) $ is an increasing function of $n$, therefore
$$
\frac{1}{\sqrt{2}}<g_1(1) \leq c_n \leq \lim_{n\to\infty}c_n<\frac{13}{14}. 
$$
Since $g_n$ is concave, we have for all $x<n$ $g_n^\prime(x)>g_n^\prime(n)=-g_{n}^{2}(n)$, and inserting this inequality in \eqref{eq:difeqgn}, we obtain
\begin{equation}\label{eq:lowbd-gn}
 g_{n}^{2}(x)\leq 1-\frac{x^2}{n^2} + \frac{x}{n^{5/3}}c_n^2, 
\end{equation}
as announced. This also proves the first part of property (iv). On $(0,n)$, $k_{n}$ is strictly positive, as $x\to\BY{n}{x}$ is negative
and increasing until $y_{n,1}$, since from \eqref{eq:Dixon-JnYn}  $y_{n,1}^{(1)}>y_{n,1}$. 
The asymptotic development 
\[
k_{n}(x)=1-\frac{x^{2}}{2}\frac{1}{n(n-1)}+O(x^4)
\]
shows that $k_{n}$ initially decreases.  Note that $k_n$ satisfies  
\begin{equation}\label{eq:difeqkn}
k_{n}^{\prime}(x) = \frac{n}{x}\left(k_{n}^{2}(x)-1+\frac{x^2}{n^2}  \right)
\end{equation}
Since $k_n(n)>0$, there exists $\kappa_n<n$ such that $k_{n}(\kappa_n)=\sqrt{1-\frac{\kappa_{n}^{2}}{n^2}}$, and $k_n$ increases on $(\kappa_n,y_{n,1})$.
The lower bound on $\kappa_n$ will be proved later. We now address property (v). The Wronskian identity can be written
\[
k_{n}+g_{n}=\frac{2}{\pi n\left(-\BJ{n}{x}\BY{n}{x}\right)},
\]
It is shown in \cite{BOYD-DUNSTER-86} that for all $x\leq n$, and $n>0$,
$$
2 \pi \left(-\BJ{n}{x}\BY{n}{x}\right)\sqrt{n^2-x^2} \leq 2.09.
$$
The lower bound on $k_n$,
$k_n \geq \kappa^{+}\sqrt{1-\left(\frac{x}{n}\right)^{2}} - g_n(x)$
follows immediately. 
To derive an upper bound, we therefore need to estimate $k_n(n)$. Using the Wronskian identity, we can write
$$
n^{1/3}k_{n}(n) =-n^{1/3}g_{n}(n)+\frac{2}{\pi\left(n^{1/3}\BJ{n}{n}\right)^{2}}\left(-\frac{\BJ{n}{n}}{\BY{n}{n}}\right)
$$
Note that $n^{1/3}g_{n}(n),n^{1/3}\BJ{n}{n}$ and $-\BJ{n}{n}/\BY{n}{n}$ are bounded increasing functions
of $n$, see \cite[8.54,855]{WATSON-25} and \cite{MULDOON-SPIGLER-84}. Therefore
\begin{eqnarray}  
\frac{4}{5}\leq	\lim\limits_{n \to \infty}
\left(-n^{1/3} g_{n}(n)+\frac{2}{\pi\left(n^{1/3}\BJ{n}{n}\right)^{2}} \left(-\frac{\BJ{1}{1}}{\BY{1}{1}}\right)\right)
&\leq& n^{1/3}k_{n}(n),  \label{eq:bd-knn-low}\\ 
\frac{7}{6} \geq -g_1(1) +\frac{2}{\pi \BJ{1}{1}^2} \lim\limits_{n\to\infty}\left(-\frac{\BJ{n}{n}}{\BY{n}{n}}\right)&\geq& n^{1/3}k_{n}(n). \label{eq:bd-knn-up}
\end{eqnarray}
We have obtained
$$
k_n\leq\max\left(\sqrt{1-\left(\frac{x}{n}\right)^{2}},\frac{7}{6}\frac{1}{n^{1/3}}\right),
$$
as announced. We can verify by inspection that $k_1>3/5$ on $(0,1)$. 
Let us compute a lower bound for $n\geq2$. Note that we have obtained that $k_n(n)<1$. 
We compute
\begin{eqnarray*}
 k_n(n) - k_n(\kappa_n) =\int_{\kappa_n}^n k_n^\prime(x)dx 
&=& 
\int_{\kappa_n}^n \frac{n}{x}\left(k_n(x)^2 -1+ \frac{x^2}{n^2}\right) dx, \\
&\leq&\left[\frac{x^2}{2n} + n(kn(n)^2-1)\ln(x) \right]_{x_1}^n\\
&=& \frac{n}{2} \left( k_n(\kappa_n)^2 + (1-k_n(n)^2) \ln\left(1-k_n(\kappa_n)^2 \right)\right) \\
&\leq& \frac{nk_n(n)^2}{2} k_n(\kappa_n)^2.
\end{eqnarray*}
This implies, using the bounds $7/6>k_n(n) n^{1/3}>4/5$,
\begin{eqnarray*}
\min\limits_{0\leq x\leq1} (k_n(x))= k_n(\kappa_n) &\geq& \frac{1}{nk_n^2(n)}\left(-1+\sqrt{1+2nk_n^3(n)}\right) \\ 
                                                   &\geq& \frac{1}{n^{1/3}} \min_{4/5\leq x \leq 7/6} \left(\frac{-1+\sqrt{1+2x^3}}{x^2} \right)\\
                                      &\geq& \frac{3}{5n^{1/3}}.
\end{eqnarray*}
Let us show now that $\kappa_n \geq n- \frac{4}{5} n^{1/3}$, to conclude the proof of Property (iv). 
Differentiating the identity \eqref{eq:difeqkn} we obtain, when $k_n^\prime\geq 0$,
\begin{eqnarray*}
k_n^{(2)}(x) &=& -\frac{1}{x} k_n ^\prime(x) +  \frac{n}{x} \left( 2\frac{x}{n^2} +2 k_n^\prime(x) k_n(x) \right) \\
             &=& \frac{2}{n} + \frac{k_n ^\prime(x)}{x} \left( 2 n k_n(x) -1 \right) \\
             &\geq& \frac{2}{n} + \frac{k_n ^\prime(x)}{x} \left( \frac{6}{5} n^{2/3}  -1 \right) \\
             &\geq& \frac{2}{n}.
\end{eqnarray*}
Therefore, we can write using the upper and lower bound on $k_n$ and the lower bound on $k_n^{(2)}$, 
$$
\left(\frac{7}{6}-\frac{3}{5}\right)\frac{1}{n^{1/3}} 
\geq k_n(n) - k_n(\kappa_n) 
=  \int_{\kappa_n}^n k_n^\prime(t)dt 
\geq \frac{2}{n} \int_{\kappa_n}^n (t- \kappa_n)dt = \frac{1}{n} (n-\kappa_n)^2. 
$$
Consequently,
$$
n- \kappa_n \leq \sqrt{\frac{17}{30}} n^{1/3} \leq \frac{4}{5} n^{1/3},
$$
which is the announced bound. 

Finally, let us address property (vi.). Note that for $x\geq x_1$, $k_n/g_n$ is increasing, as the quotient of an increasing function over a decreasing one, 
therefore the lower bound is in the interval $[0,x_1]$. Thus, the maximum is either at $x=0$ or $x=x_1$, and
$$
\max_{0\leq x\leq x_1}(k_n/g_n)\leq\max(1,k_n(n)/g_n(n))\leq \frac{7\sqrt{2}}{6}<\frac{5}{3}.
$$
Using the differential equations \eqref{eq:difeqkn} and \eqref{eq:difeqgn}, we obtain
$$
\left(\frac{k_{n}}{g_{n}}\right)^{\prime}  = \frac{n}{x g_{n}}
\left(1-\frac{x^{2}}{n^2}- k_{n} g_{n}\right)\left(1+\frac{k_{n}}{g_{n}}\right),
$$
An expansion around zero shows that
$$
1-\frac{x^{2}}{n^2}-k_{n}g_{n}=\frac{1}{n^{2}-1}\frac{x^{2}}{n^2}+O(x^4),
$$
therefore $k_n/g_n$ initially decreases. Since $k_n(n)/g_n(n)>0$, it decreases until $x_2<x_1$ such that $k_n(x_2)g_n(x_2)=1-x_2^2/n^2$, and increases afterwards.
Using the upper bound on $g_n$, we obtain, that, at that point, 
$$
\frac{k_{n}\left(x_{2}\right)}{g_{n}\left(x_{2}\right)}  =\frac{1-x_{2}^{2}/n^2}{g_{n}\left(x_{2}\right)^{2}}
 \geq
 \min\limits_{0\leq x\leq x_1} \left(1+g_{n}^{2}(n)\frac{x}{n}\left(1-\frac{x^{2}}{n^2}\right)^{-1}\right)^{-1} 
 \geq 
 \left(1+\frac{g_{n}^{2}(n)}{k_n(\kappa_n)^2} \right)^{-1} 
 \geq \frac{2}{5},
$$
as claimed.
\end{proof}

\begin{prop}\label{pro:logYn}
For any $n\neq0$, let $\zeta_n$ be the first positive solution of
$$ 
\left(\ln\left|Y_{n}\right|\right)^{(2)}\left(\zeta_n \right)=0.
$$
On $(0,\zeta_n)$, $x\to x/k_n(x)$ is increasing, and $\zeta_n$ is the maximum of $x\to x/k_n(x)$ on $(0,n)$ (resp. on $(0,y_{0,1})$) when $n\geq1$ (resp. $n=0$). 

When $n\geq 1$, we have $\zeta_n> \kappa_n$. Introducing $\chi_n := n-\frac{4}{5} n^{1/3}$ for $n\geq1$, and $\chi_1 = 1/2$, we have
\begin{equation}\label{eq:bdkn}
\frac{x}{n k_n(x)} \geq \begin{cases}
                         \left(\frac{n^2}{x^2}-1\right)^{-1/2} & \mbox{ when } 0\leq x\leq \chi_n, \\
                          \left(\frac{n^2}{\chi_n^2}-1\right)^{-1/2} & \mbox{ when } \chi_n\leq x \leq n.
                        \end{cases}
\end{equation}
When $n=0$, 
$$
\zeta_0 \approx 0.3135,\mbox{ and } \frac{\zeta_0}{k_0(\zeta_0)} \approx 0.3524
$$
In terms of previously defined functions, it is the unique solution of 
$$
k_{0}(\zeta_0)=\frac{1}{2}+\frac{1}{2}\sqrt{1-4\zeta_0^{2}}.
$$
\end{prop}
\begin{proof}
From \eqref{eq:difeqkn} we deduce that $x\to x/k_n(x)$ is a solution of the differential equation
$$
\frac{d}{dx}\left(\frac{x}{k_n}(x)\right) = \frac{ n k_n(x)( 1- n  k_n(x)) +n^2 -x^2}{n k_n^2(x)}
$$
Clearly, while $x\leq \kappa_n$, that is, while $k_n$ is decreasing, $x\to x/k_n(x)$ is increasing. We note that  $n k_n(x)( 1- n  k_n(x)) +n^2 -x^2$ only has one root greater than
$1/n$. From the lower bound on $k_n > 3/5 n^{-1/3}$ given by Proposition~\ref{pro:propsg} and by inspection for $n=1,2$, 
we verify that $(x^{-1}{k_n}(x))^\prime$ cancels at most once on $(\kappa_n,n)$.
We find
$$
\frac{d}{dx}\left(\frac{x}{k_n}\right)(\kappa_n) = \frac{ 1}{k_n(\kappa_n)} >0.
$$
Using the lower estimate on $k_n(n)$ given by \eqref{eq:bd-knn-low}  and  by inspection for $n=1,2$, we find  
$$
\frac{d}{dx}\left(\frac{x}{k_n}\right)(n) = \frac{ 1 - nk_n(n)}{k_n(n)} <0.
$$
Thus, there exists a unique maximum for $x/k_n(x)$ on $(0,n)$. Noting that $k_n(x)/x= - (\ln \left|Y_n\right|)^\prime(x)$, we conclude that this maximum is $\zeta_n$.
For any $x\in[\kappa_n,n]$,  we obtain that
$$
\frac{x}{n k_n(x)} \geq \min\left(\frac{\kappa_n}{n k_n(\kappa_n)},\frac{6 n^{1/3}}{7} \right)\geq \left(\frac{n^2}{\chi_n^2}-1\right)^{-1/2},
$$
where we used the upper bound for $x\leq\kappa_n$ given by Proposition \ref{pro:propsg}
$$
k_n(x) \leq \sqrt{1 - \frac{x^2}{n^2}},
$$ 
and where $\chi_n =  n - \frac{4}{5} n^{1/3}$ is the lower bound $\kappa_n$ given by the same proposition. In the case $n=1$, $\kappa_1\approx0.52>\frac{1}{2}$.
Altogether, we have obtained
$$
\frac{x}{n k_n(x)} \geq \begin{cases}
                         \left(\frac{n^2}{x^2}-1\right)^{-1/2} & \mbox{ when } 0\leq x\leq \chi_n, \\
                          \left(\frac{n^2}{\chi_n^2}-1\right)^{-1/2} & \mbox{ when } \chi_n\leq x \leq n.
                        \end{cases}
$$
For $n=0$, we compute that
\[
\frac{d}{dx}\left(\frac{1}{x}k_{0}\right)=1+\frac{1}{x^{2}}k_{0}\left(k_{0}-1\right)
\]
Therefore $x\to \frac{1}{x}k_{0}(x)$ is decreasing until $\zeta_0$, given by $k_{0}(\zeta_0)=\frac{1}{2}+\frac{1}{2}\sqrt{1-4\kappa_0^{2}},$
and increasing afterwards. 
\end{proof}

We conclude this section by a property of $x\to |\BH{0}{x}|$ which is useful for broadband estimates.
\begin{lem}\label{lem:logconcave}
For any $x>0$, the function $x\to \ln |\BH{0}{x}|$ is convex. Furthermore, for any $y>1$ 
$$
x\to \frac{\left|\BH{0}{xy}\right|}{\left|\BH{0}{x}\right|}
$$
is decreasing on $(0,\infty)$, and
$$
1\geq \frac{\left|\BH{0}{xy}\right|}{\left|\BH{0}{x}\right|} \geq \frac{1}{\sqrt{y}}.
$$
\end{lem}

\end{document}